\documentclass[12pt,reqno]{amsart}
\usepackage{amsaddr, latexsym, amsfonts, amsmath, amsthm, amssymb, amscd,
color}
\usepackage{pstricks}
\usepackage{pst-optic}
\usepackage{caption}
\usepackage{listings}
\usepackage{graphics}    
\usepackage{graphicx}
\usepackage{color}
\usepackage{mathrsfs}
\usepackage{xparse}
\usepackage{enumerate}
\usepackage{epstopdf}

\makeatletter
\newsavebox\myboxA
\newsavebox\myboxB
\newlength\mylenA

\newcommand*\xoverline[2][0.75]{%
    \sbox{\myboxA}{$\m@th#2$}%
    \setbox\myboxB\null
    \ht\myboxB=\ht\myboxA%
    \dp\myboxB=\dp\myboxA%
    \wd\myboxB=#1\wd\myboxA
    \sbox\myboxB{$\m@th\overline{\copy\myboxB}$}
    \setlength\mylenA{\the\wd\myboxA}
    \addtolength\mylenA{-\the\wd\myboxB}%
    \ifdim\wd\myboxB<\wd\myboxA%
       \rlap{\hskip 0.5\mylenA\usebox\myboxB}{\usebox\myboxA}%
    \else
        \hskip -0.5\mylenA\rlap{\usebox\myboxA}{\hskip
0.5\mylenA\usebox\myboxB}%
    \fi}
\makeatother
\newtheorem*{thm*}{Theorem}
\newtheorem{thm}{Theorem}
\newtheorem{prop}[thm]{Proposition}

\newtheorem{defn}[thm]{Definition}
\newtheorem{lem}[thm]{Lemma}

\theoremstyle{remark}
\newtheorem{rmk}{Remark}[section]

\numberwithin{equation}{section}
\numberwithin{thm}{section}
\def\Pf{\operatorname{Pf}}

\def\sgn{\operatorname{sgn}}

\def\det{\operatorname{det}}

\newcommand{\Hreg}[3]{\ensuremath{H_{#1,#2}\setminus(\triangleright_{#3}
\cup\triangleleft_{#3})}}
\newcommand{\Vreg}[3]{\ensuremath{V_{#1,#2}\setminus\triangleright_{#3}} }
\newcommand{\HPreg}[3]{\ensuremath{\xoverline{H}_{#1,#2}
^+\setminus(\triangleright_ { #3 }
\cup\triangleleft_{#3})}}
\newcommand{\HMreg}[3]{\ensuremath{H_{#1,#2}^-\setminus(\triangleright_{#3}
\cup\triangleleft_{#3})}}

\newmuskip\pFqskip
\mathchardef\pFcomma=\mathcode`, 

\newcommand*\pFq[5]{%
  \begingroup
  \begingroup\lccode`~=`,
    \lowercase{\endgroup\def~}{\pFcomma\mkern\pFqskip}%
  \mathcode`,=\string"8000
  {}_{#1}F_{#2}\biggl[\genfrac..{0pt}{}{#3}{#4};#5\biggr]%
  \endgroup
}

\begin{document}

  \title{Three Interactions of Holes in Two Dimensional Dimer
systems}
 
 \author[T. Gilmore]{Tomack
Gilmore$^\dagger$}
\address{Fakult\"at f\"ur Mathematik der Universit\"at Wien,\\
Oskar-Morgernstern-Platz 1, 1090 Wien, Austria.}
\email{tomack.gilmore@univie.ac.at}
\thanks{$\dagger$Research
supported by the Austrian Science Foundation (FWF), grant F50-N15, in the
framework of the Special Research Program ``Algorithmic and Enumerative
Combinatorics''.}

   \begin{abstract}
Consider the unit triangular lattice in the plane with origin
$O$, drawn so that one of the sets of lattice lines is vertical. Let $l$ and
$l'$ denote respectively the vertical and horizontal lines that
intersect $O$. Suppose the plane contains a pair of triangular holes of side
length two,
distributed symmetrically with respect to $l$ and $l'$, and oriented so that
both holes point toward $O$. Unit rhombus tilings of three different
regions of the plane are considered, namely: tilings of the entire plane;
tilings of the half plane that lies to the left of $l$ (where $l$ is considered
a free boundary, so unit rhombi are allowed to protrude halfway across it);
and tilings of the half plane that lies just below the fixed boundary $l'$.
Asymptotic expressions for the interactions of the triangular holes in these
three different regions are obtained, providing further evidence for Ciucu's
ongoing program that seeks to draw parallels between gaps in dimer systems on
the hexagonal lattice and certain electrostatic phenomena. 

 \end{abstract}
 \maketitle
\section{Introduction}\label{sec:Intro}
The study of interactions between holes (or gaps) in dimer systems was
established by Fisher and Stephenson~\cite{FishSteph} in 1963 with their results
concerning interactions of holes in dimer systems on the square
lattice. Three specific types of interaction were presented: the interaction
between a pair of dimer holes (that is, the interaction of two fixed dimers
within a
dimer system); the interaction between a pair of non-dimer holes (that is, the
interaction between two fixed monomers in a dimer system); and the interaction
of a dimer hole with a constrained boundary.

Kenyon~\cite{Keny} generalised the first of these results to an arbitrary number
of dimer holes in dimer systems on both the square and hexagonal lattices.
In~\cite{Amoeb} Kenyon, Okounkov and Sheffield further extended this work to
general planar bipartite lattices. The interaction of a monomer with a straight
line boundary in the square lattice is also due to Kenyon~\cite[Section
7.5]{Keny2}, while the interaction of a family of holes with a
constrained straight line boundary on the hexagonal lattice is tackled by
Ciucu in~\cite[Theorem 2.2]{Ciu05}. It is worth noting here that since the
triangular lattice arises as the ``dual'' of the hexagonal lattice, dimer
coverings of subgraphs of the hexagonal lattice are often much more easily
thought of in terms of unit rhombus
tilings of sub-regions of the triangular lattice in the plane (so monomers
correspond to
unit triangles and dimers correspond to unit rhombi). A dimer covering of the
entire hexagonal lattice then corresponds to a tiling of the
entire plane by unit rhombi (sometimes referred to as a sea of unit rhombi).

Ciucu has also studied various types of interactions between
non-dimer holes in dimer systems on the hexagonal lattice
(equivalently, interactions between
triangular
holes within a
sea of unit
rhombi~\cite{Ciu01}\cite{MihaiPN}\cite{Ciu05}\cite{Ciu02}\cite{Ciu09}),
thereby
establishing close analogies to certain electrostatic phenomena.
Ciucu conjectures that interactions between holes in dimer
systems on the hexagonal lattice are governed by the laws of two dimensional
electrostatics (that is, Coulomb's laws). More specifically, the conjecture
states that by taking the exponential of the negative of the electrostatic
energy of the two dimensional system of physical charges obtained by considering
each hole as a point charge of size and magnitude specified by a
statistic\footnote{For a triangular hole $t$ on the unit
triangular lattice, $q(t)$ is defined to be the signed difference between the
number of left and right-pointing unit triangles that comprise $t$.} $q$,
one may recover (up to a multiplicative constant) asymptotic expressions for the
dimer-mediated interactions of holes on the hexagonal lattice. This conjecture
remains wide open, though it has been shown to hold for two fairly general
families of holes.

In light of this, further types of interaction have been studied. Recently
in~\cite{KrattFree}, Ciucu and Krattenthaler determined the interaction between
a left-pointing triangular hole of side length two and a vertical free boundary
that borders a sea of unit rhombi on the left half of the plane. It would appear
that this result (which inspired the present paper) is the
first treatment of such an interaction in the literature.
Theorem~\ref{thm:VertCor} below is a direct analogy of the interaction presented
in~\cite{KrattFree} in the sense that here the hole has
been ``flipped'' and instead points toward the free boundary. Curiously,
Theorem~\ref{thm:VertCor} agrees with that of~\cite{KrattFree}
only up to a multiplicative constant.

As well as Theorem~\ref{thm:VertCor}, this paper establishes two other types of
interaction between a pair of triangular holes of side length two that point
towards each other in
dimer systems (that is, a pair of triangles oriented so that one left-pointing
triangle
lies directly to the right of one right-pointing triangle). These
interactions are: the interaction between such a pair of holes that
lie within a sea of unit rhombi (Theorem~\ref{thm:CompCorr}); and the
interaction
between such a pair of holes that lie on a fixed boundary of a sea of
unit rhombi (Theorem~\ref{thm:HorzCor}). All three of these results agree with
the original conjecture of Ciucu and so serve as yet more evidence that
interactions between holes in dimer systems are governed by the laws of two
dimensional electrostatics. A brief outline of the paper follows.

Section~\ref{sec:SetUp} discusses the
correspondence between
dimer systems on the hexagonal lattice and rhombus tilings of hexagons on the
triangular lattice, and also
states precisely the family of holey hexagons (that is, hexagons containing
triangular holes) that are of particular interest. Certain symmetry classes of
rhombus
tilings of a holey hexagon correspond precisely to rhombus tilings of
smaller sub-regions of the same hexagon. The three different types of
interaction that are the focus of this
work arise from considering tilings of these
sub-regions as they become infinitely large. These sub-regions
are defined in Section~\ref{sec:SetUp}, along with the corresponding
interactions (or correlation functions). Asymptotic expressions for
these interactions are also stated here (Theorem~\ref{thm:CompCorr},
Theorem~\ref{thm:VertCor},
and Theorem~\ref{thm:HorzCor}), which together comprise the main
results. Proposition~\ref{thm:MatchFac}-- a factorisation result that is
frequently used throughout this paper-- is also stated in this section.

Section~\ref{sec:EnumForm} briefly discusses the relationship between rhombus
tilings of hexagons and plane partitions. This is followed by four exact
enumeration formulas that count tilings of sub-regions of holey
hexagons (Theorem~\ref{thm:VertExact},
Theorem~\ref{thm:HorzExact}, Theorem~\ref{thm:HorzWExact}, and
Theorem~\ref{thm:HExact}). Proofs of these theorems follow in
Sections~\ref{sec:NonInt}
and~\ref{sec:eval}.

In Section~\ref{sec:NonInt} the
well-known bijection between rhombus tilings and non-intersecting lattice paths
is combined with existing results due to Stembridge~\cite{Stemb} and
Lindstr\"{o}m-Gessel-Viennot~\cite{LGV} in order to express the number of
tilings of different regions as either a Pfaffian
(Proposition~\ref{prop:VertMat}) or a determinant
(Proposition~\ref{prop:HorzMat}, Proposition~\ref{prop:HorzWMat}). In
Section~\ref{sec:eval} the matrices defined in Proposition~\ref{prop:VertMat}
are manipulated using standard row and column operations in such a way that
the Pfaffian of each matrix may be expressed as a determinant of a much smaller
matrix. The unique $LU$-decompositions of these smaller matrices, along with
those from Proposition~\ref{prop:HorzMat} and Proposition~\ref{prop:HorzWMat},
are then stated
and proved (Theorem~\ref{thm:VertLU},
Lemma~\ref{thm:HorzWLU}, Theorem~\ref{thm:VertOddLU}, and
Theorem~\ref{thm:HorzLU}), from which the exact
enumeration formulas in Section~\ref{sec:EnumForm} immediately follow.

Section~\ref{sec:Asym} contains proofs of the main results. These follow from
inserting the enumeration formulas from Section~\ref{sec:EnumForm} into the
correlation functions defined in Section~\ref{sec:SetUp}. Asymptotic
expressions for the three types of interaction are then derived by studying
these correlation functions as the distance between the holes (or the distance
between the hole and the free boundary) becomes very large.

\section{Set-Up and Results}\label{sec:SetUp}

Consider the hexagonal lattice $\mathscr{H}$ in terms of its dual
$\mathscr{T}$ (that is, the planar triangular lattice consisting of unit
equilateral triangles drawn so that one of the sets of lattice lines is
vertical). Then
monomers on $\mathscr{H}$ correspond to unit triangles on $\mathscr{T}$ and
dimers on $\mathscr{H}$ correspond to unit rhombi on $\mathscr{T}$ (a dimer on
$\mathscr{H}$ is equivalent to joining two unit triangles that share
exactly one edge on $\mathscr{T}$). It follows that dimer coverings of regions
of
$\mathscr{H}$ that contain gaps correspond exactly to rhombus tilings of
regions of $\mathscr{T}$ that contain triangular holes.

Fix some origin $O$ in $\mathscr{T}$ and denote by $H_{a,b}$ the hexagonal
region centred at $O$ with side lengths $a,\,b,\,a,\,a,\,b,\,a$ (going clockwise
from the southwest edge). Suppose $T$ is a set of fixed unit
triangles contained within the interior of
$H_{a,b}$, where the distances between the triangles are parametrised by
$k$. Denote by $H_{a,b}\setminus T$ the hexagonal region $H_{a,b}$ with the set
of triangles $T$ removed from its interior (such regions will often be
referred to as \emph{holey hexagons}, a term coined by Propp in~\cite{Propp}).
Suppose further that $\xi>0$
is real such that $b\thicksim\xi a$ is integral. Then the \emph{interaction}
(otherwise known as the \emph{correlation function}, or simply
the \emph{correlation}) between the holes in $T$ is defined to be
$$\omega(k;\xi)=\lim_{a\to\infty}\frac{M(H_{a,b}\setminus
T)}{M(H_{a,b})},$$
where $M(R)$ denotes the number of rhombus tilings of the region $R$. Different
types of interactions may be obtained by replacing $H_{a,b}\setminus T$ and
$H_{a,b}$ with specific sub-regions of
$H_{a,b}\setminus T$ (and
$H_{a,b}$ respectively) in the above formula. Note that in order for a rhombus
tiling of $H_{a,b}$ to exist (and hence for the above definition to make
sense), it must be the case that $\sum_{t\in T}q(t)=0$, where $q(t)$ is the
statistic defined in Section~\ref{sec:Intro}. The triangular holes and
sub-regions of particular interest are defined below.

For some non-negative integer $k$, let $\triangleright_k$ denote the
right-pointing equilateral triangle of side length two at lattice distance $k$
to
the left of the origin such that the centre of its vertical side lies on the
horizontal symmetry axis of $H_{a,b}$. Define $\triangleleft_k$ analogously,
that is, $\triangleleft_k$ denotes the left-pointing triangle of side length two
with its vertical edge
at lattice distance $k$ to the
right of the origin. Then $\Hreg{a}{b}{k}$
denotes the hexagonal region $H_{a,b}$ from which two inward pointing
triangular
holes of side length two have been removed that are symmetrically distributed
with
respect to the
vertical and horizontal symmetry axes of $H_{a,b}$. Figure~\ref{fig:HexNoDim}
(left) shows
such a
region for $a=7,\,b=5$ and $k=4$.
\begin{rmk}\label{rmk:Parity}
Under this construction the parity of $a$ and $b$ determine the parity
of $k$, that is, if $a$ and $b$ are of the same parity then $k$ is even,
otherwise $k$ is odd.
\end{rmk}
\begin{figure}[t]
\centering
\includegraphics[scale=0.45]{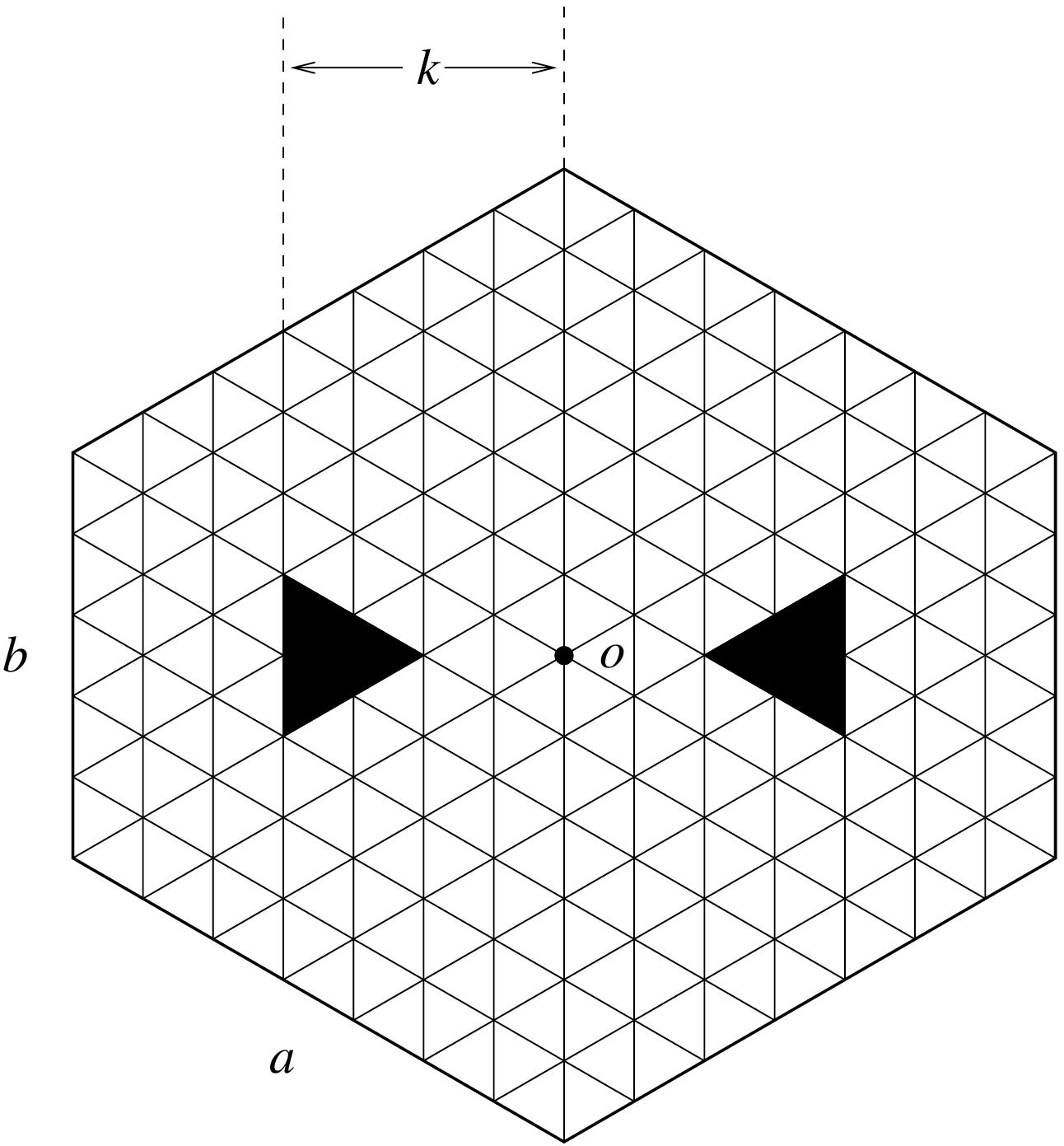}\qquad \includegraphics[scale=0.45]{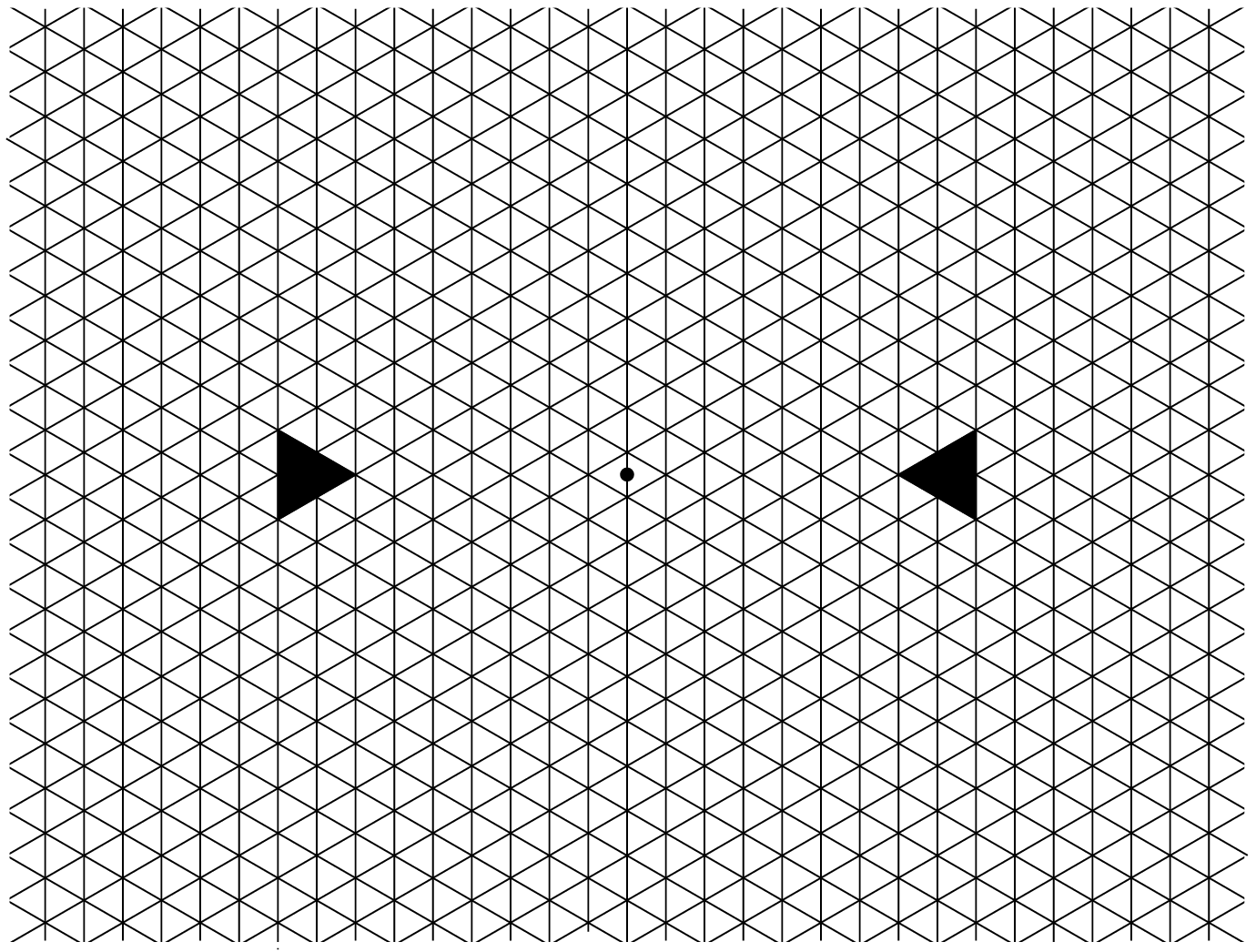}
\caption {Left: the holey hexagon
\Hreg{7}{5}{4}. Right: a small section of the plane of unit triangles given by
$\Hreg{a}{b}{9}$ as $a$ tends to infinity.}
\label{fig:HexNoDim}
\end{figure}
Consider $\Hreg{a}{b}{k}$ as $a$ is sent to infinity. It should be clear that
this gives the entire plane of unit triangles
containing two triangular holes with their
vertical sides at lattice distance $2k$ apart, a portion of which may be seen in
Figure~\ref{fig:HexNoDim} (right). Hence the correlation function
$$\omega_{H}(k;\xi)=\lim_{a\to\infty}\frac{M(\Hreg{a}{b}{k})}{M(H_{a,b})}$$
is the interaction (in terms of $k$) between a pair of triangular holes of side
length 2
oriented so that they point directly at each other within a sea of unit
rhombi.
\begin{thm}\label{thm:CompCorr}
As $k$ becomes very large the interaction defined above, $\omega_H(k;\xi)$, is
asymptotically
$$\omega_{H}(k;\xi)\thicksim\left(\frac{1}{
2k\pi } \left(\frac{\xi+1}{2}\right)^{2k-2}\right)^2.$$
\end{thm}

Denote by $\Vreg{a}{b}{k}$ the left half of
$\Hreg{a}{b}{k}$ whose rightmost
boundary is the vertical line $l$ that intersects the origin. If $l$ is a free
boundary (that is, unit rhombi are permitted to protrude across $l$),
then the number of rhombus tilings of $\Vreg{a}{b}{k}$ is clearly
equal to the number of
vertically symmetric tilings of
$\Hreg{a}{b}{k}$. Figure~\ref{fig:VertSym} (left) shows
such a tiling for $\Hreg{6}{6}{4}$ along with the corresponding tiling of
$\Vreg{6}{6}{4}$ (centre).

As $a$ tends to infinity, the region
$\Vreg{a}{b}{k}$ gives the left half plane of unit
triangles, constrained on the right by $l$, and containing
a right-pointing triangular hole with its vertical side at lattice distance $k$
to the left of $l$ (see Figure~\ref{fig:VertSym}, right). Hence the correlation
function
$$\omega_{V}(k;\xi)=\lim_{a\to\infty}\frac{M(\Vreg{a}{b}{k})}{M(V_{a,b})}$$
gives the interaction (in terms of $k$) between a right-pointing triangle of
side length two and a vertical free boundary that borders a sea of unit rhombi
that tile the left half plane.

\begin{thm}\label{thm:VertCor}
 As $k$ becomes very large the interaction defined
above, $\omega_V(k;\xi)$, is asymptotically
$$\omega_{V}(k;\xi)\thicksim\frac{\sqrt{\xi(\xi+2)}}{2\pi
k}\left(\frac{(\xi+1)}{ 2}\right)^{2(k-1)}.$$
\end{thm}
\begin{figure}[t!]
\includegraphics[scale=0.4]{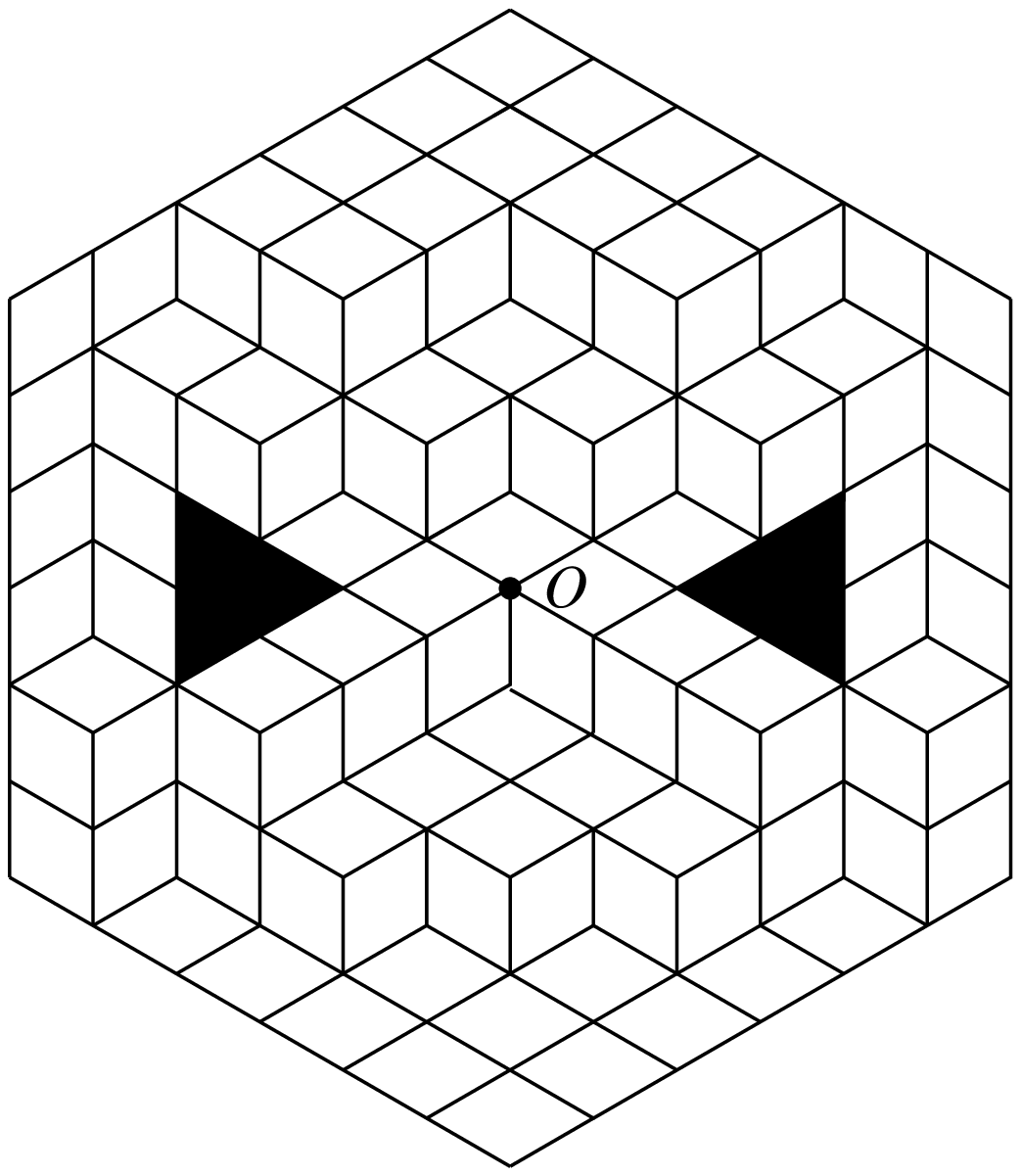}\qquad\qquad
 \includegraphics[scale=0.4]{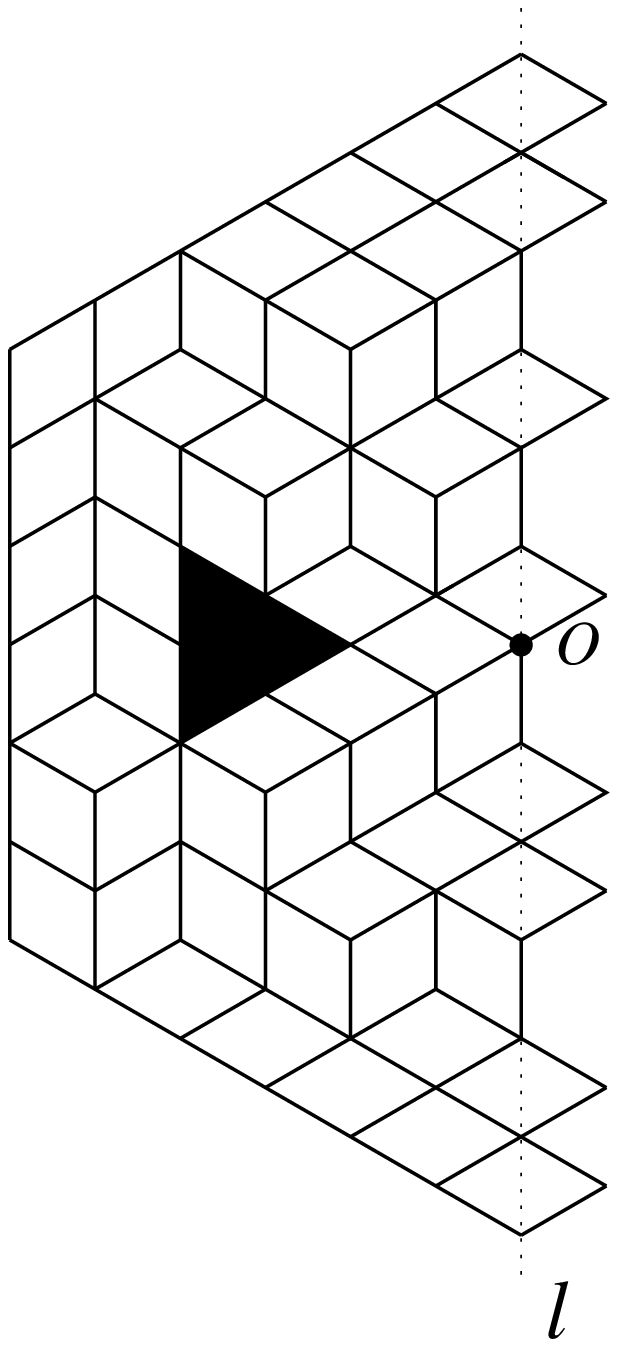}\includegraphics[scale=0.4]{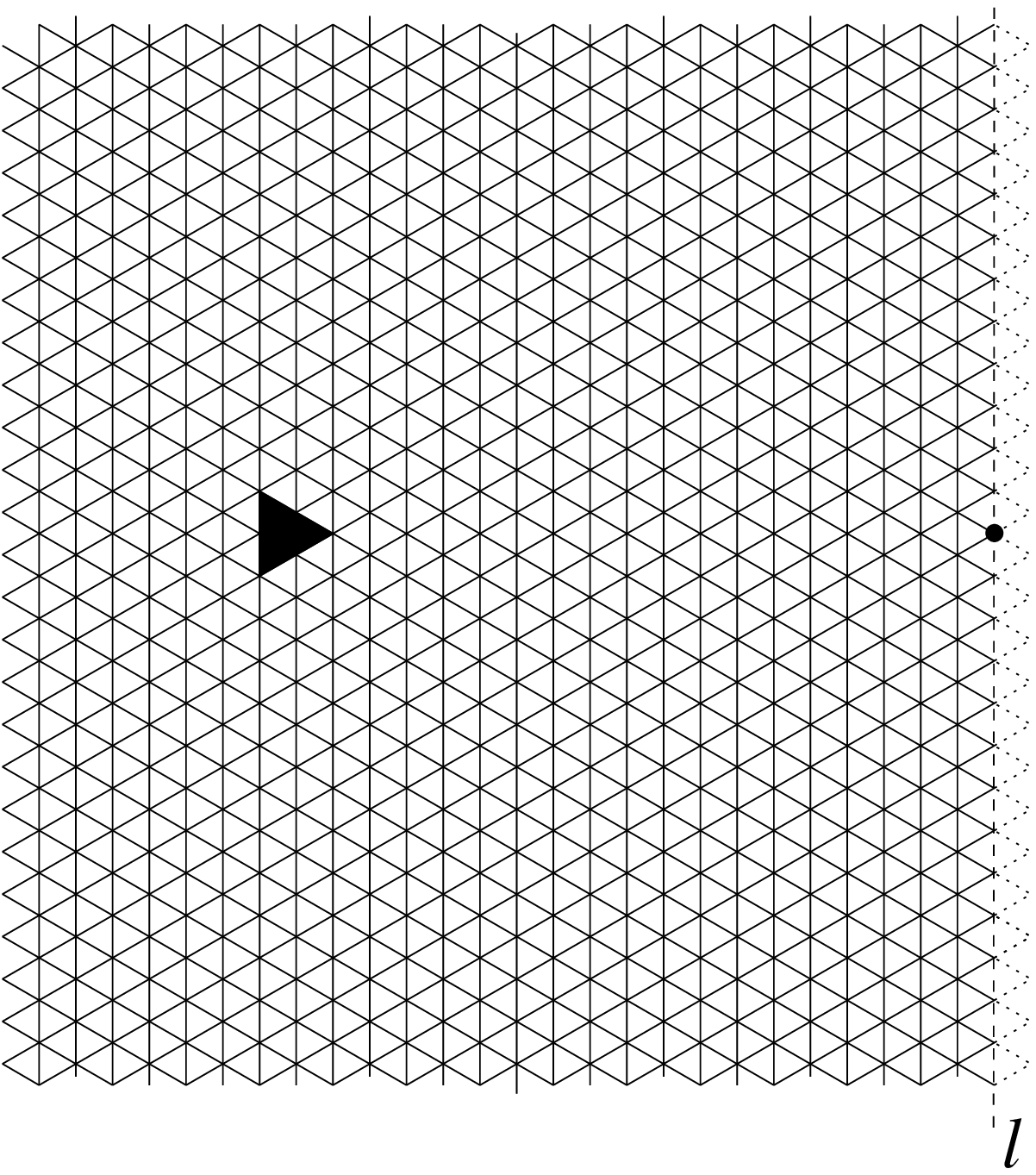}
\caption { Left: a vertically
symmetric rhombus tiling of
$\Hreg{6}{6}{4}$. Centre: the corresponding tiling
of $\Vreg{6}{6}{4}$. Right: a small section of the left half of the plane
of unit triangles given by $\Vreg{a}{b}{20}$ as
$a$ tends to infinity.}
\label{fig:VertSym}
\end{figure}
Consider now the two regions obtained by cutting
$\Hreg{a}{b}{k}$ along the
zig-zag line that lies just below its horizontal
symmetry axis. Denote the upper and lower portions of
$\Hreg{a}{b}{k}$ with
$H_{a,b}^+\setminus(\triangleright_{k}\cup\triangleleft_k)$
and $\HMreg{a}{b}{k}$ respectively. Figure~\ref{fig:HalfHalf} (left) shows the
upper and lower regions
of $\Hreg{7}{8}{5}$. If $b$ is even then rhombus tilings of
$\HMreg{a}{b}{k}$ correspond to horizontally symmetric tilings of the entire
region $\Hreg{a}{b}{k}$.

Now consider $\HMreg{a}{b}{k}$ as $a$ is sent to infinity. Clearly one obtains
the lower half plane
of unit triangles, bounded above by a fixed boundary on which lie two
triangular holes with vertical sides at lattice distance $2k$ apart-- see
Figure~\ref{fig:HalfHalf} (right). The correlation function
$$\omega_{H^-}(k;\xi)=\lim_{a\to\infty}\frac{M(\HMreg{a}{b}{k})}{M(H^-_{a,b})}$$
then gives the interaction (in terms of $k$) of a pair of triangular holes of
side length two that point directly towards each other and are positioned along
the fixed upper boundary that borders unit rhombus tilings of the lower half
plane.

\begin{thm}\label{thm:HorzCor}
 As $k$ becomes very large the correlation function
defined above, $\omega_{H^-}(k;\xi)$, is asymptotically equivalent to that of
$\omega_{V}(k;\xi)$ from
Theorem~\ref{thm:VertCor}.
\end{thm}
\begin{figure}[t]
 \includegraphics[scale=0.4]{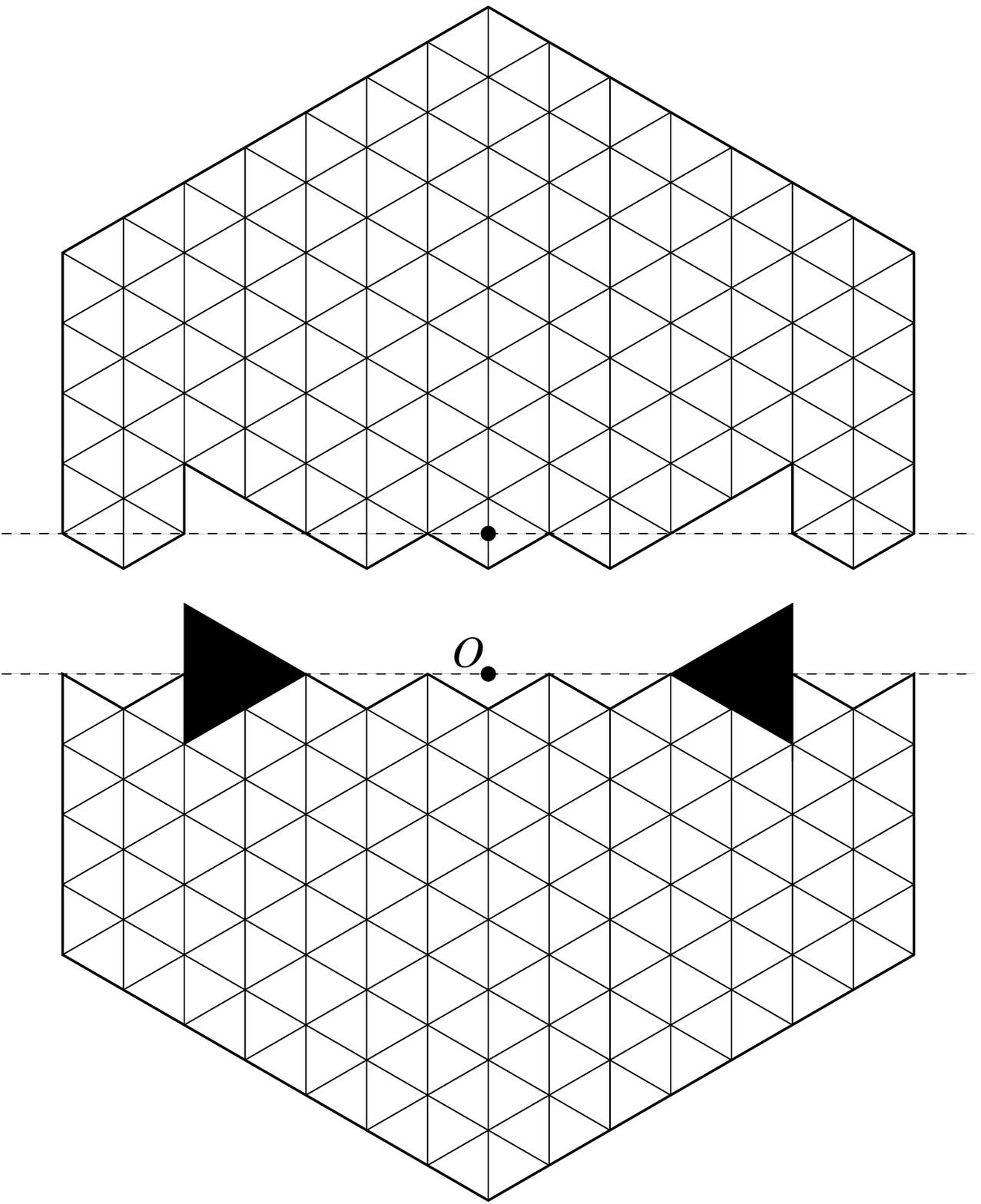}\quad\includegraphics[scale=0.6]{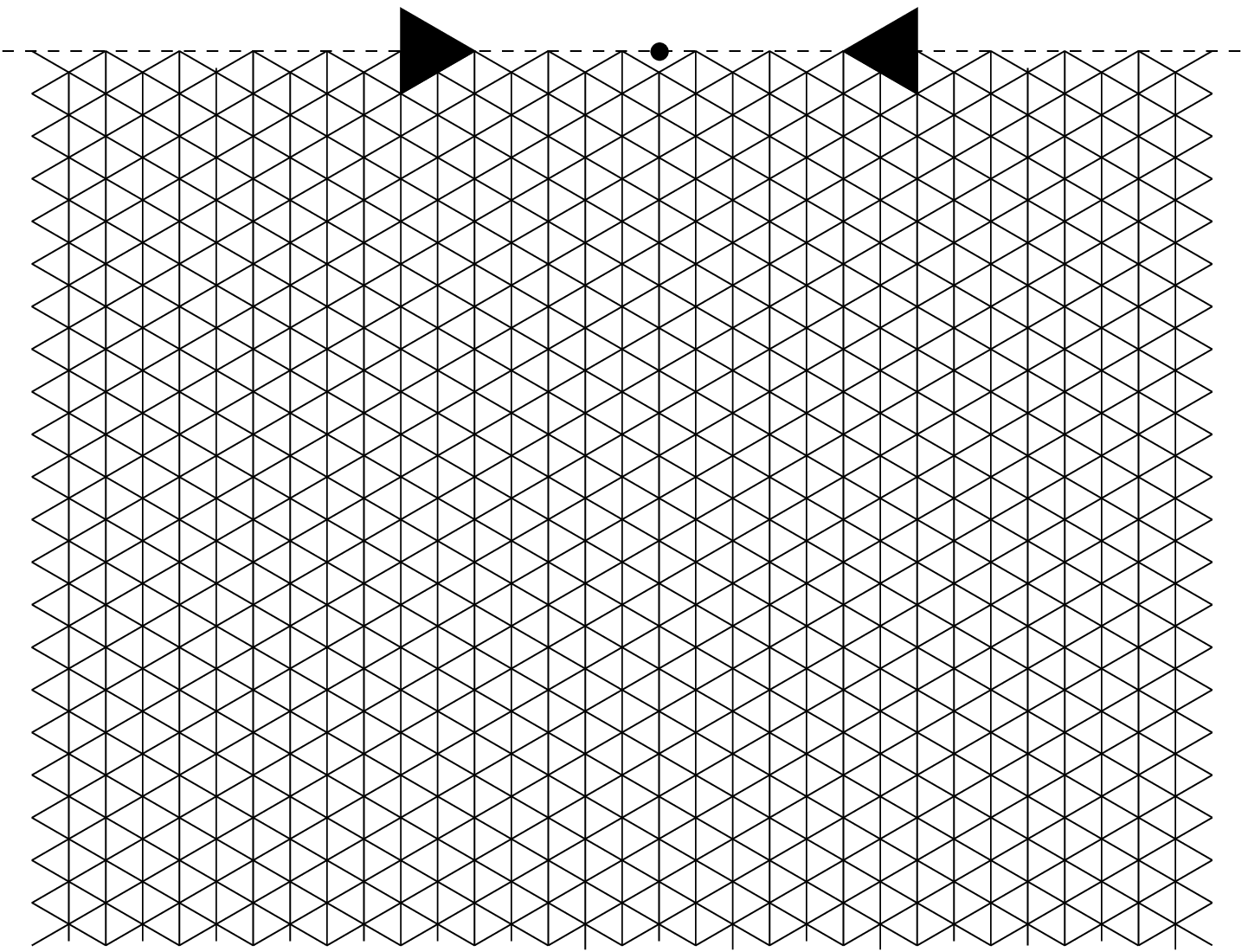}
\caption{Left: The regions
$H_{7,8}^+\setminus(\triangleright_4\cup\triangleleft_4)$ (upper) and
$\HMreg{7}{8}{4}$ (lower). Right: a small section of the lower half of the
plane of unit triangles given by $\HMreg{a}{b}{7}$ as
$a$ tends to infinity.}
\label{fig:HalfHalf}
\end{figure}

Observe that for $\xi\neq 1$, Theorem~\ref{thm:CompCorr} and
Theorem~\ref{thm:VertCor} give distorted dimer statistics.
Indeed, these expressions blow up or vanish exponentially for $\xi>1$ and
$\xi<1$ respectively. So when considering the three types of interaction
detailed previously, it makes sense to consider only regular hexagons, that is,
hexagons with all sides the same length. In light of this the above asymptotic
expressions may be re-packaged in the following way:
\begin{align}\label{eq:phys1}
 \omega_V(k;1)\thicksim&\frac{3}{4\pi\textrm{d}(\triangleright_k,O)}\\
\label{eq:phys2}
 \omega_H(k;1)\thicksim&\left(\frac{\sqrt{3}}{\pi\textrm{d}(\triangleright_k,
\triangleleft_k) } \right)^2
\end{align}
where d$(a,b)$ denotes the Euclidean distance between $a$ and $b$.

The first
result above is analogous to that of Ciucu and Krattenthaler~\cite{KrattFree},
which presents an asymptotic result for the correlation of a triangular hole
that has been
``flipped'' and instead points away from the vertical free boundary $l$. Within
an
infinitely large sea of rhombi one could be forgiven for thinking that the
orientation of such a hole would have no effect on the interaction between the
hole and the free boundary, yet this result shows that such an interaction
differs by a factor of 3. Any intuitive reason as to
why this phenomenon should occur remains a complete mystery. 

Both of these results are entirely in keeping with the original conjecture of
Ciucu~\cite[Conjecture 1]{MihaiPN}. The asymptotic expression for
$\omega_{H}(k;\xi)$ may be obtained (up to a multiplicative constant) by
considering the exponential of the
negative of the electrostatic energy of the system of physical charges obtained
by viewing the triangular holes as point charges, where one has signed magnitude
$+2$ and the other $-2$ according to the definition of the statistic $q$ in
Section~\ref{sec:Intro}. In a similar way the asymptotic expression for
$\omega_{V}(k;\xi)$ shows that a triangular gap is attracted to a vertical free
boundary $l$ in precisely the same way that an electrical charge is attracted to
a straight line conductor when placed near it. Hence ~\eqref{eq:phys1} and
~\eqref{eq:phys2} above further support the hypothesis that interactions of gaps
in the hexagonal lattice are governed by Coulomb's laws for two dimensional
electrostatics.

This section concludes with an important proposition concerning the enumeration
of
tilings of the entire region
\Hreg{a}{b}{k}. It follows from
the Matchings Factorization Theorem~\cite{MatchFac} by employing arguments that
are almost identical to those of Ciucu and Krattenthaler~\cite{KrattFact}.

\begin{prop}\label{thm:MatchFac} The total number of tilings of the region
\Hreg{a}{b}{k} has the factorisation
$$M(\Hreg{a}{b}{k})=M(\HMreg{a}{b}{k})\cdot
M_w(H_{a,b}^+\setminus(\triangleright_k\cup\triangleleft_k)),$$
where $M(H)$ denotes the number of tilings of the region $H$ and
$M_w(H_{a,b}^+\setminus(\triangleright_k\cup\triangleleft_k))$ denotes the
weighted count of the region
$H_{a,b}^+\setminus(\triangleright_k\cup\triangleleft_k)$ where the unit rhombi
that lie along its bottom edge carry a weight of $2$.
\end{prop}

\begin{rmk}
 From now on, denote by \HPreg{a}{b}{k} the region
$H_{a,b}^+\setminus(\triangleright_k\cup\triangleleft_k)$ with the added
condition that the unit rhombi lying along its bottom edge have a weight of 2,
so that
$M_w(H_{a,b}^+\setminus(\triangleright_k\cup\triangleleft_k)=M(\HPreg{a}{b}{k}
)$.
\end{rmk}

This proposition implies that rhombus tilings of
$\Hreg{a}{b}{k}$ may be enumerated by counting separately the number of
tilings of $\HPreg{a}{b}{k}$ and
$\HMreg{a}{b}{k}$ and taking their product. If $b$ is an even integer then
tilings of the latter region correspond to horizontally symmetric tilings of the
entire hexagon $\Hreg{a}{b}{k}$. By enumerating tilings of
$\Hreg{a}{b}{k}$ in terms of this factorisation formula one obtains, for free,
an enumeration formula for horizontally symmetric tilings of $\Hreg{a}{b}{k}$. 

A similar correlation function involving $\HPreg{a}{b}{k}$ may also be defined,
namely
$$\omega_{H^+}(k;\xi)=\lim_{a\to\infty}\frac{M(\HPreg{a}{b}{k}
)}{M(\xoverline{H}^+_{a,b})},$$
which gives the interaction between a pair of triangular holes of side length
two that lie on the
fixed weighted boundary of a sea of unit rhombi. According to~\cite{KrattFact},
the number of rhombus tilings of
$\xoverline{H}^+_{a,b}$ is equal to the number of tilings of $V_{a,b}$ and so
this correlation function becomes

$$\omega_{H^+}(k;\xi)=\lim_{a\to\infty}\frac{M(\HPreg{a}{b}{k}
)}{M(V_{a,b})}.$$

Theorem~\ref{thm:VertCor} above is obtained by
directly analysing $\omega_V(k;\xi)$ as $k$ becomes large, whereas
Theorem~\ref{thm:CompCorr} follows from considering the
product of $\omega_{H^+}(k;\xi)$ and $\omega_{H^-}(k;\xi)$ asymptotically. In
order to do so, it is necessary to first find exact enumeration formulas for
$M(\Vreg{a}{b}{k})$, $M(\HMreg{a}{b}{k})$, and $M(\HPreg{a}{b}{k})$.
Sections~\ref{sec:EnumForm},~\ref{sec:NonInt} and~\ref{sec:eval} are concerned
with stating and proving these exact formulas, while the
proofs of Theorems~\ref{thm:CompCorr},~\ref{thm:VertCor}
and~\ref{thm:HorzCor} may be found in Section~\ref{sec:Asym}.

\section{Some Exact Enumeration Formulas}\label{sec:EnumForm}
The enumeration of rhombus tilings (in various forms and guises) has a fairly
long history beginning with MacMahon~\cite{CombAnal} in the early
20th century. Although at the time MacMahon was interested in counting the
number of plane partitions whose Young tableaux fit inside an $a\times b\times
c$ box, there is a straightforward bijection between the three dimensional
representation of these objects (as unit cubes stacked in the corner of an
$a\times b\times c$ box) and two dimensional unit rhombus tilings of
$H_{a,b,c}$, the hexagon with side lengths $a$, $b$, $c$, $a$, $b$, $c$ (going
clockwise from the southwest edge).
Under this same bijection symmetric plane partitions whose Young tableaux fit
inside
an $a\times a\times b$ box correspond to vertically symmetric rhombus tilings of
$H_{a,b}$. Similarly, transpose-complementary plane partitions whose
Young tableaux fit inside an $a\times a\times 2b$ box correspond to
horizontally symmetric tilings of $H_{a,2b}$.

In~\cite[p. 270]{CombAnal} MacMahon conjectured a formula for the
weighted enumeration of symmetric plane partitions of height at most $b$ whose
Young
Tableaux fit inside an $a\times a\times b$ box. Under some specialisation,
MacMahon's
conjectured formula counts vertically symmetric rhombus tilings of $H_{a,b}$.
Andrews managed to prove MacMahon's conjecture in full in~\cite{PlanePart1},
and hence established the (ordinary) enumeration of symmetric plane partitions.
As far as ordinary enumeration goes, however, Andrews was not the first to
prove the corresponding result. It was Gordon who, around 1970 (though
published much later in~\cite{Gordon}), proved the so-called Bender-Knuth
conjecture~\cite{BKCon}, which for a special case gives the ordinary
enumeration of symmetric plane partitions. Since then further refinements and
alternative proofs have been found by Macdonald~\cite[pp. 83--85]{McDon},
Proctor
\cite[Prop. 7.3]{Proct}, Fischer~\cite{Ilse2} and Krattenthaler~\cite[Theorem
13]{Kratt2}. This theorem states (in an equivalent form) that the number of
vertically symmetric rhombus tilings of $H_{a,b}$ (that is, the number of
rhombus tilings of the region $V_{a,b}$) is
$$ST(a,b)=\prod_{i=1}^{a}\frac{2i+b-1}{2i-1}\prod_{1\le
i<j\le a}\frac{i+j+b-1}{i+j-1}.$$
A similar result holds for rhombus tilings of $\Vreg{a}{b}{k}$.
\begin{thm}\label{thm:VertExact}
Suppose $k$ is a non-negative integer.
\begin{enumerate}[(i)]
 \item The number of
vertically symmetric tilings of $\Hreg{n}{2m}{k}$ is
$$\left[\sum_{s=1}^{m}B_{n,k}(s)\cdot D_{n,k}(s)
\right]\times ST(n,2m),$$
where
\begin{align*}
B_{n,k}(s)&=\frac{(-1)^{s+1} (-k+n+1)! (n+s-1)! (n+2s-1)!
\left(\frac{k}{2}+\frac{n}{2}+s-2\right)!}{(s-1)!
\left(\frac{n}{2}-\frac{k}{2}\right)! \left(\frac{k}{2}+\frac{n}{2}-1\right)! (2
n+2 s-1)! \left(-\frac{k}{2}+\frac{n}{2}+s\right)!},\\
 D_{n,k}(s)&=\frac{(-1)^{s+1} (2 s-2)! (n-k)! (n+s-1)!
\left(\frac{k}{2}+\frac{n}{2}+s-2\right)!}{(s-1)!
\left(\frac{n}{2}-\frac{k}{2}\right)! \left(\frac{k}{2}+\frac{n}{2}-1\right)!
(n+2s-2)! \left(-\frac{k}{2}+\frac{n}{2}+s\right)!}\\&\quad+\frac{2 (-1)^{s+1}
(2s-2)! (n-k)! (n+s)! \left(\frac{k}{2}+\frac{n}{2}+s-2\right)!}{(s-2)!
\left(\frac{n}{2}-\frac{k}{2}\right)! \left(\frac{k}{2}+\frac{n}{2}\right)! (n+2
s-2)! \left(-\frac{k}{2}+\frac{n}{2}+s\right)!}.
\end{align*}
\item The number of vertically symmetric tilings $\Hreg{n}{2m-1}{k}$ is
$$ST(n,2m-1)\times\left[\sum_{s=1}^{m-1}\left(E_{n,k}^*(s)-\frac{D_
{n}^*(s)E_{n,k }^*(m) } {D_{n}^*(m)}\right)\cdot
B_{n,k}^*(s)\right]$$
where
\begin{align*}
 B_{n,k}^*(j)&=\frac{2 (-1)^{j+1} (j+n-1)! (2 j+n-1)! (n-k)! \left(\frac{1}{2}
(2 j+k+n-5)\right)!}{(j-1)! (2 j+2 n-1)! \left(\frac{1}{2} (-k+n-1)\right)!
\left(\frac{1}{2} (k+n-3)\right)! \left(\frac{1}{2} (2
j-k+n+1)\right)!}\\&+\frac{4 (-1)^{j+1} (j+n)! (2 j+n-1)! (n-k)!
\left(\frac{1}{2} (2 j+k+n-5)\right)!}{(j-2)! (2 j+2 n-1)! \left(\frac{1}{2}
(-k+n-1)\right)! \left(\frac{1}{2} (k+n-1)\right)! \left(\frac{1}{2} (2
j-k+n+1)\right)!}\\
D_{n}^*(i)&=\frac{2^{n}(2 i-2)! (i+n-1)!}{(i-1)! (2 i+n-2)!},\\
E_{n,k}^*(i)&=\frac{(-1)^{i+1} (2 i-2)! (i+n-1)! (n-k)! \left(\frac{1}{2} (2
i+k+n-3)\right)!}{(i-1)! (2 i+n-2)! \left(\frac{1}{2} (-k+n-1)\right)!
\left(\frac{1}{2} (k+n-1)\right)! \left(\frac{1}{2} (2
i-k+n-1)\right)!}.
\end{align*}
\end{enumerate}
\end{thm}

The following formula, due to
Proctor~\cite{Proct}, counts the number of transpose-complementary plane
partitions that fit inside an $a\times a\times 2b$ box:

$$TC(a,2b)=\binom{a+b-1}{a-1}\prod_{i=1}^{a-2}\prod_{j=i}^{a-2}\frac{2b+i+j+1}
{i+j+1}.$$
Since transpose complementary plane partitions correspond to horizontally
symmetric tilings of $H_{a,b}$, the above formula counts precisely the
number of rhombus tilings of the region $H^-_{a,2b}$. Again, a similar
result holds for rhombus tilings of $\HMreg{a}{2b}{k}$.
\begin{thm}\label{thm:HorzExact}The number of horizontally symmetric tilings of
$\Hreg{n}{2m}{k}$ (equivalently, the number of tilings of the region
$\HMreg{n-1}{2m+1}{k}$) is given by
 $$\left[\sum_{s=1}^{m}B_{n,k}'(s)\cdot
D_{n,k}'(s)\right]\cdot TC(n,2m),$$
where
\begin{align*}
 B'_{n,k}(s)&=\frac{(-1)^{s+1} (n-k)! (n+s-2)! (n+2 s-1)! \left(\frac{1}{2}
(k+n+2 s-4)\right)!}{2 (s-1)! \frac{n-k}{2}! \left(\frac{1}{2} (k+n-2)\right)!
(2 n+2 s-3)! \left(-\frac{k}{2}+\frac{n}{2}+s\right)!},\\
D'_{n,k}(s)&=\frac{(-1)^{s+1} (2 s)! (n-k)! (n+s-1)!
\left(\frac{1}{2} (k+n+2 s-4)\right)!}{2( s!) \frac{n-k}{2}!
\left(\frac{1}{2} (k+n-2)\right)! (n+2 s-2)!
\left(-\frac{k}{2}+\frac{n}{2}+s\right)!}.
\end{align*}
\end{thm}
\begin{rmk}
 Observe that if $b$ is odd there exist no horizontally symmetric tilings of
the region $H_{a,b}$. 
\end{rmk}

The following enumerative result counts the number of tilings of the
weighted region $\HPreg{a}{b}{k}$ defined in the previous section.

\begin{thm}\label{thm:HorzWExact}
 The number of rhombus tilings of
$\HPreg{n}{2m}{k}$, the upper half
of
the region $\Hreg{n}{2m}{k}$ containing
lozenges of
weight 2 lying along the horizontal symmetry axis is 
$$\left[\displaystyle\sum_{s=1}^{m}B_{n,k}(s)\cdot
E_{n,k}(s)\right]\cdot ST(n,2m),$$
where $B_{n,k}(s)$ is defined as in Theorem~\ref{thm:VertExact} and
$$E_{n,k}(s)=\frac{(-1)^{s+1} (2 s-2)! (-k+n+1)! (n+s-1)!
\left(\frac{k}{2}+\frac{n}{2}+s-2\right)!}{(s-1)!
\left(\frac{n}{2}-\frac{k}{2}\right)! \left(\frac{k}{2}+\frac{n}{2}-1\right)!
(n+2 s-2)! \left(-\frac{k}{2}+\frac{n}{2}+s\right)!}.$$
Similarly, the number of tilings of $\HPreg{n+1}{2m-1}{k}$ is equal to
$2\cdot M(\HPreg{n}{2m}{k})$.
\end{thm}

\begin{rmk}
By re-writing the above result in terms of Theorem~\ref{thm:VertCor} and
combining it with Theorem~\ref{thm:HorzExact} via
Proposition~\ref{thm:MatchFac}, one may construct a factorisation for
$\Hreg{a}{b}{k}$ that is in a sense similar to that of Ciucu and
Krattenthaler~\cite{KrattFact}. Indeed, it was the observation of this
factorisation that lead to the results presented in the current work.
\end{rmk}

MacMahon proved in~\cite{CombAnal} that the total number of plane
partitions of an $a\times b\times c$ box (equivalently, the number of rhombus
tilings of the hexagon $H_{a,b,c}$) is given by the formula
$$T(a,b,c)=\prod _{i=1}^a \prod _{j=1}^b \prod _{k=1}^c
\frac{i+j+k-1}{i+j+k-2}.$$
Combining Theorem~\ref{thm:HorzExact} with Theorem~\ref{thm:HorzWExact}
via Proposition~\ref{thm:MatchFac} gives a similar formula for enumerating
tilings of $\Hreg{a}{b}{k}$. 

\begin{thm}\label{thm:HExact}
 The total number of tilings of the region
$\Hreg{n}{2m}{k}$ is
 \begin{equation*}
  \left[\sum_{s=1}^{m}B_{n,k}'(s)\cdot
D_{n,k}'(s)\right]\times\left[\displaystyle\sum_{t=1}^{m}B_{n,k}(t)\cdot
E_{n,k}(t)\right]\times
T(n,2m,n),
 \end{equation*}
 whilst the number of tilings of $\Hreg{n}{2m-1}{k}$ is
 $$\left[\displaystyle\sum_{s=1}^{m}B_{n-1,k}(s)\cdot
E_{n-1,k}(s)\right]\times\left[
\sum_ { t=1 } ^ { m-1 } B_ { n+1 , k }'(t)\cdot
D_{n+1,k}'(t)\right]\times T(n,2m-1,n).$$
\end{thm}

\begin{rmk}
 If $k=n$, then the first formula in the above theorem gives the number of
rhombus tilings of the hexagon $H_{n,2m}$ containing two unit triangular dents
in each vertical side, positioned symmetrically with respect to the horizontal
symmetry axis. This is a special case of the more general class of
rhombus tilings considered by Ciucu and
Fischer in~\cite{Ils4}.
\end{rmk}

The following two sections are dedicated to proving Theorem~\ref{thm:VertExact},
Theorem~\ref{thm:HorzExact} and Theorem~\ref{thm:HorzWExact} by translating
sets of rhombus tilings into families of non-intersecting paths
(Section~\ref{sec:NonInt}). The enumeration of these families of paths can
be expressed as the determinant (or Pfaffian) of
certain matrices. These determinants and Pfaffians are then evaluated in
Section~\ref{sec:eval}.

\section{Non-intersecting Lattice Paths}\label{sec:NonInt}
The aim of this section is to express the number of rhombus tilings of
$\Vreg{a}{b}{k}$, $\HPreg{a}{b}{k}$ and $\HMreg{a}{b}{k}$
as determinants (or Pfaffians) of particular matrices.

Consider first the region $\Vreg{a}{b}{k}$ described in
Section~\ref{sec:SetUp}. Tilings of this region may be represented by unique
families of
lattice paths across dimers (see Figure~\ref{fig:DimerPath}, left), which
may
in turn be translated into unique families of non-intersecting lattice paths
consisting of north and east unit steps on the integer lattice with origin $O'$
(see Figure~\ref{fig:DimerPath}, right).
Counting rhombus tilings of $\Vreg{a}{b}{k}$ is therefore equivalent to counting
the number of families $(P_1,P_2,...,P_{b})$ of non-intersecting lattice paths
between a set
of starting points $A=\{(-s,s):1\le s\le b\}$ and a set of end points
$I=\{(x,y)\in\mathbb{Z}^2:x+y=a\}\cup\{t^-,t^+\}$, with the requirement
that $t^+=((a-b-k)/2-1,(a+b-k)/2+1)$ and $t^-=((a-b-k)/2,(a+b-k)/2)$ must be
included as end
points. Figure~\ref{fig:DimerPath} shows the unique family of lattice paths
corresponding to the tiling found in Figure~\ref{fig:VertSym}.

By the well-known theorem of Stembridge~\cite{Stemb},
such families of
non-intersecting paths may be expressed in terms of a \emph{Pfaffian}, defined
below.

\begin{defn}
 The Pfaffian of a $2n\times2n$ skew-symmetric matrix $A$ is
 $$\Pf (A)=\displaystyle\sum_{\pi\in\mathcal{M}_{2n}}\sgn(\pi)\prod_
{i<j\textrm{ matched in }\pi}(A)_{i,j},$$
where $\mathcal{M}_{2n}$ denotes the set of all perfect matchings on
the
set of vertices $\{1,2,\dots,2n\}$ and sgn$(\pi)=(-1)^{\textrm{cr}(\pi)},$ where
cr$(\pi)$ is the number of crossings of $\pi$. 
\end{defn}
\begin{rmk}\label{rmk:Pfaf}It is a well known fact that 
$$\Pf(A)^2=\det(A).$$\end{rmk}
\begin{figure}[t!]
\centering
 \includegraphics[scale=0.65]{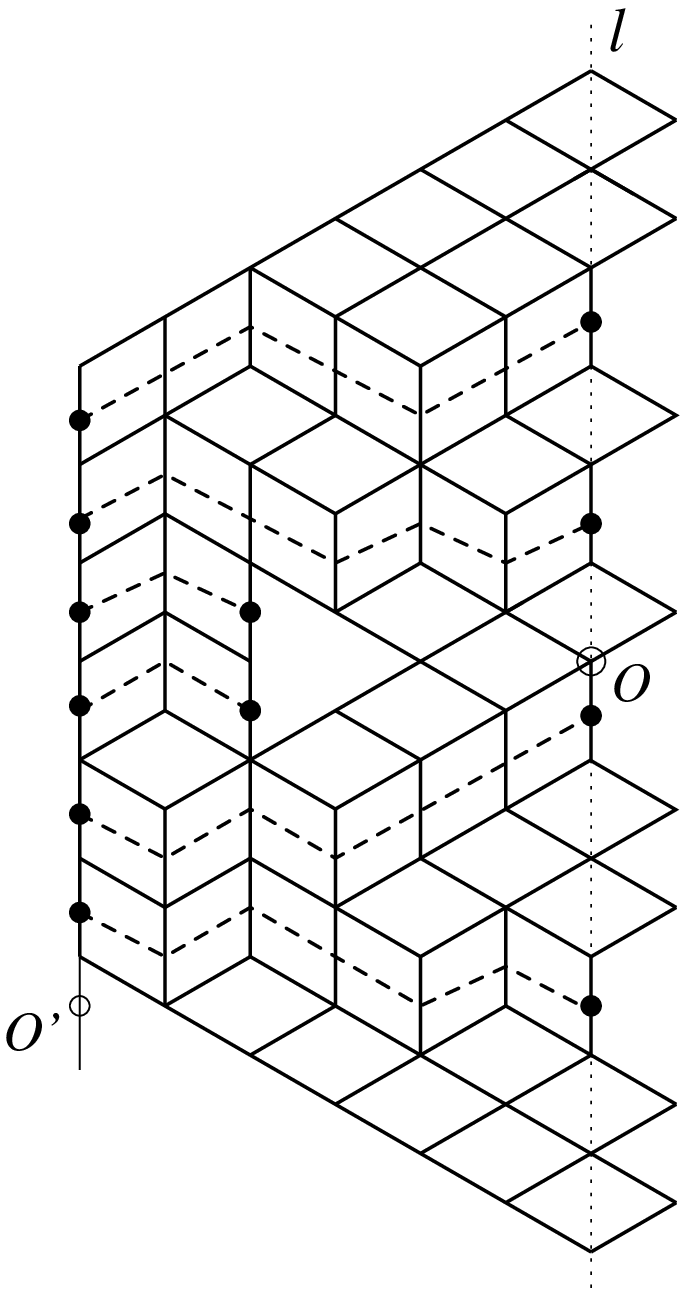}\qquad\includegraphics[scale=0.65]
{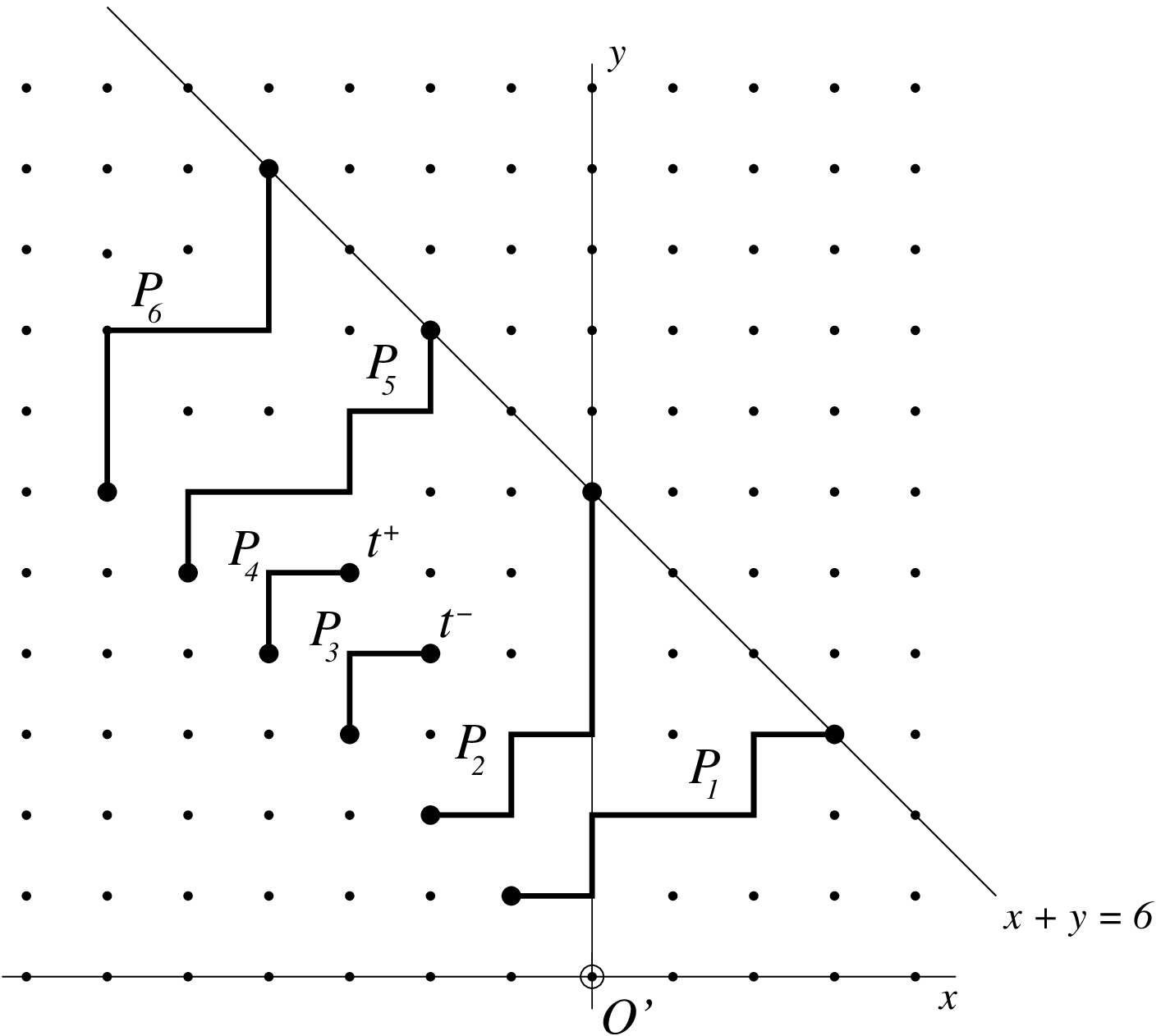} \caption{Left: the tiling from Figure~\ref{fig:VertSym} as lattice
paths
across dimers. Right: the
translation of this tiling to a family of non-intersecting paths in
$\mathbb{Z}^2$.}
\label{fig:DimerPath}
\end{figure}
\begin{thm}[Stembridge]\label{thm:StemOk}
Let $V=\{v_1,...,v_r\}$, $U=\{u_1,...,u_p\}$ and $E=\{e_1,...,e_q\}$ be
finite sets of lattice points in the integer lattice $\mathbb{Z}^2$ and assume
that $p+r$ is even. Denote by $\mathscr{P}(A\rightarrow B)$ the
total number of families of non-intersecting paths from the set $A$ to the set
$B$, whilst $P(a\rightarrow b)$ denotes the number of paths starting at point
$a$ and ending at point $b$. If $E$ is totally ordered in such a way that $U$ is
D-compatible\footnote{This is a technical condition on the order in which the
start points are connected to the end points in each family of
paths.} with $V\oplus E$ (disjoint union), then the total number of
non-intersecting paths starting at $U$ and ending at $v\oplus E$ is
$$\mathscr{P}(U\rightarrow v\oplus E)=\emph{Pf}\begin{pmatrix}
                                                   Q&Y\\-Y^t&0
                                                  \end{pmatrix},$$
where $Q=(Q_{i,j})_{1\le i,j\le p}$ is given by
$$Q_{i,j}=\sum_{1\le s<t\leq q}(P(u_i\rightarrow
e_s)P(u_j\rightarrow e_t)-P(u_i\rightarrow
e_t)P(u_j\rightarrow e_s)),$$
and $Y=(Y_{i,j})_{i\in\{1,\dots,p\},j\in\{1,\dots,r\}}$ is given by
$$Y_{i,j}=P(u_i\rightarrow v_{j}).$$
\end{thm}
\begin{rmk}\label{rmk:StemOk}
 In case $p+r$ is odd, a phantom vertex $u_{p+1}=e_{q+1}$ may be adjoined with
the property that $$P(u_s\to e_{q+1})=\begin{cases}1,& s=p+1,\\0,&
otherwise.\end{cases}$$
Then it may be shown that $\mathscr{P}(U\cup u_{p+1}\to (v\oplus E)\cup
e_{q+1})$ is given by
$$\Pf \begin{pmatrix}Q&Y\\-Y^t&0\end{pmatrix},$$ where $Y=(Y_{i,j})_{1\le i\le
p+1,1\le j\le r}$ is as above with the added condition that $Y_{p+1,j}=0$ for
all $j$, and $Q=(Q_{i,j})_{1\le i,j\le p+1}$ is given by
$$Q_{i,j}=\sum_{1\le s<t\le q+1}(P(u_i\rightarrow
e_s)P(u_j\rightarrow e_t)-P(u_i\rightarrow
e_t)P(u_j\rightarrow e_s)).$$
\end{rmk}
\begin{rmk}Throughout this article sums in which the starting index
is not necessarily smaller than the ending index are interpreted according to
the following standard convention:

\begin{equation}
 \sum_{r=m}^{n-1}\textrm{Exp}(r)=\begin{cases}\sum_{r=m}^{n-1}\textrm{Exp}
(r)&n>m\\0&n=m\\-\sum_{r=n}^{m-1}\textrm{Exp}(r)&n<m.

\end{cases}
\end{equation}
\end{rmk}

Theorem~\ref{thm:StemOk} may be used to count families of non-intersecting paths
that
correspond to vertically symmetric tilings of holey hexagons.

\begin{prop}\label{prop:VertMat}
 The number of vertically symmetric tilings
of the hexagon $\Hreg{n}{2m}{k}$ is the
Pfaffian of the skew-symmetric matrix $F=(f_{i,j})_{1\le i,j\le 2m+2}$ defined
by
$$f_{i,j}=\begin{cases}\sum_{r=i-j+1}^{j-i}{2n\choose
n+r},&1\le i< j\le 2m,\\
{n-k\choose (n-k)/2-m+i},&i\in\{1,...,2m\}, j=2m+1,\\
{n-k\choose (n-k)/2-m-1+i},&i\in\{1,...,2m\},j=2m+2,\\
0,&2m+1\le i<j\le 2m+2,\end{cases}$$
whilst the number of vertically symmetric
tilings of $\Hreg{n}{2m-1}{k}$ is the Pfaffian of the
skew-symmetric
matrix $F^*=(f^*_{i,j})_{1\le i,j\le 2m+2}$ defined by
$$f^*_{i,j}=\begin{cases}\sum_{r=i-j+1}^{j-i}{2n\choose
n+r},&1\le i< j\le 2m-1,\\2^n, &i\in\{1,...,2m-1\},
j=2m\\
{n-k\choose (n-k+1)/2-m+i},&i\in\{1,...,2m-1\}, j=2m+1,\\
{n-k\choose (n-k-1)/2-m+i},&i\in\{1,...,2m-1\},j=2m+2,\\
0,&2m+1\le i<j\le 2m+2.\end{cases}$$

\end{prop}

\begin{proof}
 Consider first the region $\Hreg{n}{2m}{k}$. Let $U=A$, $V=\{t^-,t^+\}$ and
$E=I\setminus v$, and $F$ be the resulting
matrix in Theorem~\ref{thm:StemOk}. Then by
the simple fact that
$P((a,b)\rightarrow(c,d))={c-a+d-b\choose c-a}$, it should be clear that the
proposition holds for $j=2m+1$ and $j=2m+2$. For
$1\le i<j\le 2m$,
\begin{align*}f_{i,j}&=\sum_{1\le s<t\le
2m+n}\left({n\choose s-i}{n\choose t-j}-{n\choose s-j}{n\choose
t-i}\right)\\&=\sum_{t=1}^{n+2m}\sum_{s=1}^{2m+n}\left({
n\choose n-s+i } { n\choose s+t-j}-{n\choose n-s+j}{n
\choose s+t-i}\right)\\
&=\sum_{t=1}^{n+2m}\left({2n\choose n+t+i-j}-{2n\choose
n+t+j-i}\right)\\
&=\sum_{r=i-j+1}^{j-i}{2n\choose n+r},
\end{align*}
where the Chu-Vandermonde convolution has been used in the second line.

Next consider the region $\Hreg{n}{2m-1}{k}$. Let the point
$A_{m+1}=I_{n+m+1}=(-1,0)$ be a
phantom vertex as described in Remark~\ref{rmk:StemOk}. Let $U=A\cup A_{m+1}$,
$V=\{t^-,t^+\}$ and $E=(I\setminus v)\cup I_{n+m+1}$ in Theorem
\ref{thm:StemOk} and let $F^*$ be the resulting matrix. Then $F^*$ agrees with
the matrix $F$ (the determinant of which counts tilings of $\Vreg{n}{2m}{k}$)
everywhere except in the $(2m)$-th row and column, where
 \begin{align*}f^*_{i,2m}&=\sum_{s=1}^{n+m}P(u_i\to e_s)\\&=2^n,\end{align*} and
\begin{align*}f^*_{2m,j}&=-\sum_{s=1}^{n+m}P(u_j\to e_s)\\&=-2^n,\end{align*}
for $1\le i,j\le 2m-1$. This
concludes the proof. 
\end{proof}
Consider now the region $\HMreg{n}{2m}{k}$.
Observe
that in any tiling of $\HMreg{n}{2m}{k}$ the left and right most vertical sides
comprise a set of forced rhombi (Figure~\ref{fig:HorzPath}, left, illustrates
such a forcing). Hence it follows that
$$M(\HMreg{n-1}{2m+1}{k})=M(\HMreg{n}{2m}{k}),$$ so it suffices to
enumerate tilings of the region \HMreg{n}{2m}{k}.

Consider then the region
$\Hreg{n}{2m}{k}$ for some integer $m$.
Once again, by translating sets of tilings to sets of lattice paths across
dimers, sets of tilings of \HMreg{n}{2m}{k} correspond to families of
non-intersecting paths beginning at a set of start points
$A=\bigcup_{s=1}^{m+1}A_s$, where $A_s=(s,1-s)$ for $s\in\{1,\dots,m\}$ and
$A_{m+1}=((n+k)/2+1,(n+k)/2)$, and ending at $I=\bigcup_{t=1}^{m+1}I_t$, where
$I_t=(n+t,n+1-t)$ for $t\in\{1,\dots,m\}$ and $I_{m+1}=((n-k)/2+1,(n-k)/2)$,
with the condition that no path intersects the line $y=x$ (see
Figure~\ref{fig:HorzPath}). With this labelling
the number of tilings of
$\HMreg{n}{2m}{k}$ is equal to the number
of families of non-intersecting paths $(P_1,P_2,\dots,P_{m+1})$ where $P_{i}$
denotes a path from $A_{i}$ to $I_{\sigma(i)}$ and $\sigma=(1,m+1)$ is the
permutation on $m+1$ elements.      
\begin{figure}[t]
 \centering
 \includegraphics[scale=0.65]{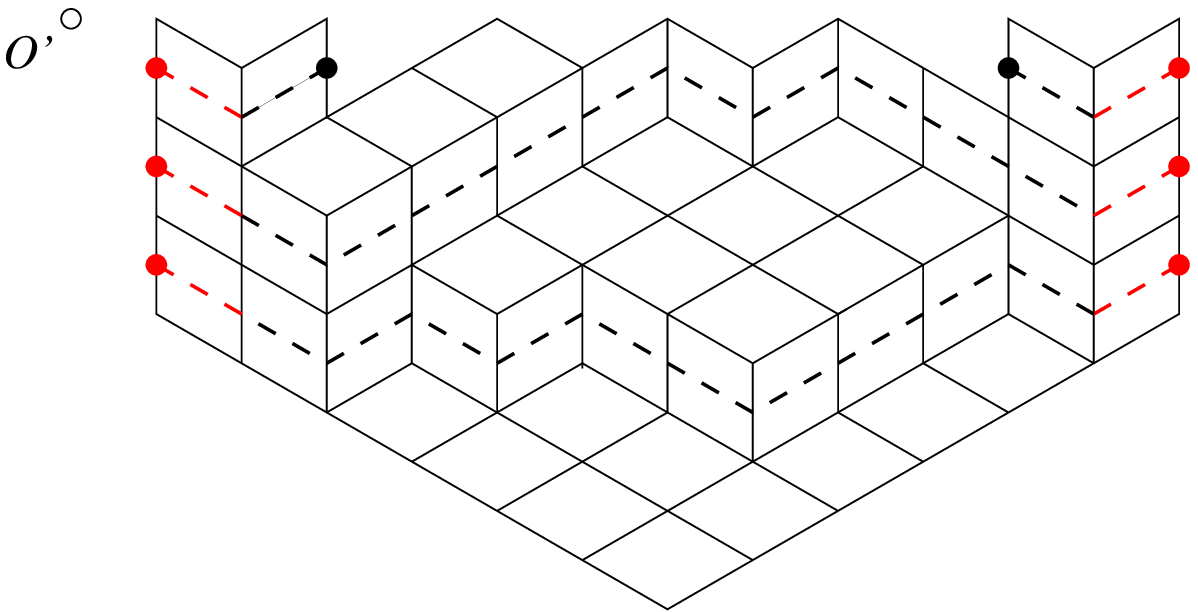}\qquad\includegraphics[scale=0.65]
{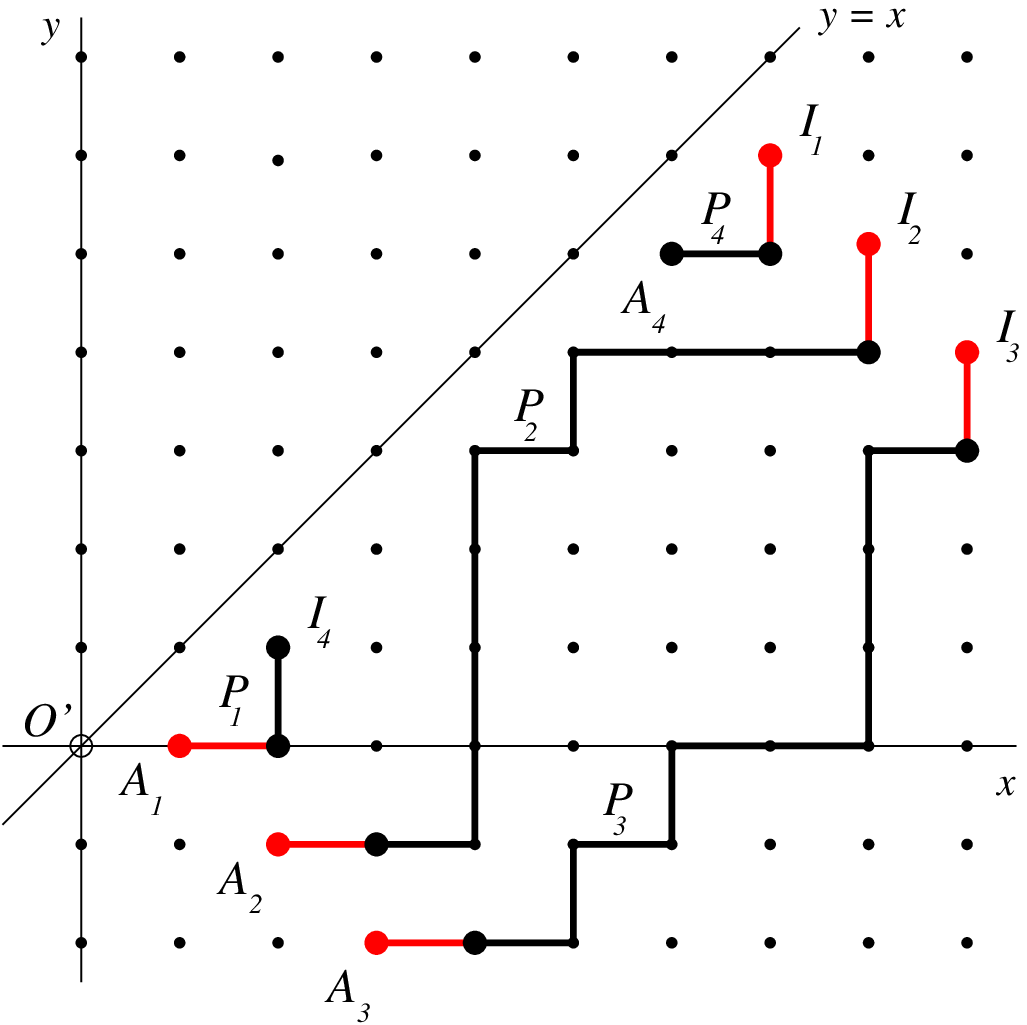} \caption{Left: A tiling of
$\HMreg{6}{6}{4}$. Right: The
corresponding set of lattice paths. Lines and dashed lines marked in red
highlight forced dimers on the vertical sides.}
\label{fig:HorzPath}
\end{figure}
Such families of paths may be counted directly using the well-known theorem of 
Lindstr\"{o}m-Gessel-Viennot~\cite{LGV}, recalled below.

Let $G=(V,E)$ be a locally directed acyclic graph, that is, every vertex has
finite degree and $G$ contains no directed cycles. Consider base vertices
$B=\{b_1,b_2,\dots,b_n\}$ and destination vertices $D=\{d_1,d_2,\dots,d_m\}$
and assign to each directed edge $e$ a weight, $\omega_e$, belonging to some
commutative ring. For each directed path $P:b\to d$, let
$\omega(P)$ be the product of the weights of the edges of the path.
For any two vertices $b$ and $d$, let $e(b,d)=\sum_{P:b\to
d}\omega(P)$ denote the sum over all paths from $b$ to $d$. If we
assign the weight 1 to each edge then $e(b,d)$ counts the number of paths
from $b$ to $d$. Let $\mathcal{P}$ denote the $n$-tuple of paths
$(P_1,P_2,\dots,P_n)$ from $B$ to $D$ that satisfy the following properties:

\begin{itemize}
 \item There exists a permutation $\sigma$ of $\{1,2,\dots,n\}$ such that for
every $i$, $P_i$ is a path running from $b_i$ to $d_{\sigma(i)}$.
\item For any $i\neq j$, the paths $P_i$ and $P_j$ have no common vertices.

\end{itemize}
Given such an $n$-tuple, denote by $\sigma(\mathcal{P})$ the permutation
from the first condition above.

Consider now the matrix $N$ with $(i,j)$-entry equal to $e(b_i,d_j)$. The
Lindstr\"{o}m-Gessel-Viennot Theorem then states that the determinant of $N$ is
the signed sum over all $n$-tuples $\mathcal{P}=(P_1,P_2,\dots,P_n)$ of
non-intersecting paths from $B$ to $D$,

$$\det(N)=\sum_{\mathcal{P}:B\to D}\sgn(\sigma(\mathcal{P}))\prod_{i=1}
^n\omega(P_i).$$

That is, the determinant of $N$ counts the weights of all sets of
non-intersecting paths from $B$ to $D$, each affected by the sign of the
corresponding permutation $\sigma$ given by $P_i$ taking $b_i$ to
$d_{\sigma(i)}$. Indeed if the weights of each edge are 1 and the only
permutation possible is the identity then $\det(N)$ counts exactly
$\mathscr{P}(B\to D)$ (that is, the number
of non-intersecting paths between $B$ and $D$).

\begin{prop}\label{prop:HorzMat}
 The number of horizontally symmetric tilings of
$\Hreg{n}{2m}{k}$ is 
$(-\det(G)),$
where $G=(g_{i,j})_{1\le i,j\le m+1}$ is the matrix given by

$$g_{i,j}=\begin{cases}
\binom{2n}{n+j-i}-\binom{2n}{n-j-i+1}, &1\le i,j\le m,\\
\binom{n-k}{(n-k)/2+1-i}-\binom{n-k}{(n-k)/2-i},&j=m+1,\,1\le i\le
m,\\
\binom{n-k}{(n-k)/2+1-j}-\binom{n-k}{(n-k)/2-j},&i=m+1,\,1\le j\le
m,\\0,&otherwise.\end{cases}$$
\end{prop}
\begin{proof}
 In the result of Lindstr\"om, Gessel and Viennot replace $B$ with the set $A$
and replace $D$ with the set $I$
(defined previously) and let $G$ be the matrix whose $(i,j)$-entry is the
number of
paths from $A_i$ to $I_j$. The only permutation that gives rise to
$(m+1)$-tuples of non-intersecting paths, $\mathcal{P}_{m+1}$, is the
permutation $\sigma=(1,m+1)$ on $m+1$ objects. By
Lindstr\"{o}m-Gessel-Viennot it follows that $\det(G)=-\mathscr{P}(A\to I)$,
where
$\mathscr{P}(X\to Y)$ is as in Theorem~\ref{thm:StemOk}. 

All that remains to be seen is that the entries of $G$ do
indeed count the number of paths between $A_i$ and $I_j$. Care must be taken to
ensure that only those paths that do not intersect the line
$y=x$ are counted.

Given a starting point $a$ and an ending point $b$, let $\mathcal{P}$
denote the set of all paths from $a$ to $b$ (including those that cross the
line $y=x$). Consider a second set of paths $\mathcal{P}'$ obtained by
taking all paths in $\mathcal{P}$ that intersect $y=x$ and reflecting portions
between touching points and the last segment of each path in the line $y=x$,
thereby obtaining all paths from $a$ to $b'$ (where $b'$ is the reflection of
$b$ in $y=x$). Then the difference $|\mathcal{P}|-|\mathcal{P}'|$
gives the total number of paths from $a$ to $b$ that never intersect the line
$y=x$.

Suppose $1\le i,j\le m$. Then the number of paths starting at $A_i$ and ending
at $I_j$ is given by $$P(A_i\to I_j)=P((i,1-i)\to
(n+j,n+1-j)).$$
Letting $I_j'$ denote the reflection of $I_j$ in the line $y=x$, the
number of paths from $A_i$ to $I_j$ that do not intersect the line
$y=x$ is then $$P(A_i\to I_j)-P(A_i\to
I_j')=\binom{2n}{n+j-i}-\binom{2n}{n-j-i+1}.$$
A similar argument holds for the other entries of $G$. This concludes
proof.
\end{proof}

Precisely the same method may be used to enumerate tilings of the region
containing weighted lozenges, $\HPreg{n}{2m}{k}$, with the exception that any
path from
$A$ to $I$ that intersects the line $y=x$ at $t$-many points has a total weight
of $2^t$. It is not hard to convince oneself that
$|\mathcal{P}|+|\mathcal{P}'|$ counts all such paths. Such an argument
appears in~\cite{KrattFact}. Proceeding in exactly the same way as
Proposition~\ref{prop:HorzMat} gives the following:
\begin{prop}\label{prop:HorzWMat}
 The number of tilings of the weighted region
$\HPreg{n}{2m}{k}$ is given by $(-\det(G^+))$, where $G^+=(g^+_{i,j})_{1\le
i,j\le m+1}$ is the
matrix defined by
$$g^+_{i,j}=\begin{cases}
\binom{2n}{n-i-j+1}+\binom{2n}{n+i-j}, &1\le i,j\le m,\\
\binom{n-k+1}{(n-k)/2+i},&j=m+1,1\le i\le m,\\
\binom{n-k+1}{(n-k)/2+j},&i=m+1,1\le j\le
m,\\0,&otherwise.\end{cases}$$
\end{prop}
\begin{rmk}\label{rmk:WHorzOdd}
 Consider the lattice path
representation of $\HPreg{n}{2m}{k}$ defined earlier. A lattice path
representation of
$\HPreg{n+1}{2m-1}{k}$ may be obtained from that of \HPreg{n}{2m}{k} by
extending each start point $A_s$ to the left by 1 and extending each end point
$I_s$ vertically by 1 for $s\in\{1,\dots,m\}$. Any tiling of
this region will contain forced lozenges along its vertical sides and so
$$M(\HPreg{n+1}{2m-1}{k})=2\cdot
M(\HPreg{n}{2m}{k}),$$where the extra factor of 2 is due to the starting point
$A_1$ lying
on the line $y=x$.
\end{rmk}

Theorems~\ref{thm:VertExact},~\ref{thm:HorzExact} and~\ref{thm:HorzWExact}
follow from evaluating the determinants and Pfaffians of the matrices defined
above. In the following section these determinants and Pfaffians are evaluated
by considering the $LU$-decomposition of certain matrices.

\section{Evaluation of Determinants}\label{sec:eval}

This section begins with an extension of Gordon's
Lemma~\cite{Gordon} that expresses the Pfaffian of the matrix $F$
defined in Theorem~\ref{prop:VertMat} as the 
determinant of a much smaller matrix. Similar manipulations reduce
the matrix $F^*$ to the same effect. The (unsigned) determinants of these two
smaller
matrices, as well as those of Proposition~\ref{prop:HorzMat}
and Proposition~\ref{prop:HorzWMat}, are then evaluated by finding their unique
$LU$-decompositions.

\begin{lem}\label{lem:Gordon} For a positive integer $m$ and a
non-negative integer $l$, let $A$ be the
$(2m+2l)\times(2m+2l)$ skew-symmetric matrix of the form
 
$$A=\begin{pmatrix}
    X & Y \\
    -Y^{t} &Z
   \end{pmatrix},$$
   for which the following properties hold:
\begin{enumerate}[(i)]
\item $X$ is a $2m\times2m$ matrix such that $X=(x_{j-i})_{1\leq i,j\leq   2m}$
and $x_{j,i}=-x_{i,j}$;
\item $Z=(z_{i,j})_{1\le i,j\le 2l}$ is a matrix satisfying
$z_{i,j}+z_{i+l,j}+z_{i,j+l}=0\textrm{ for } 1\le i,j\le l$, and
$z_{j,i}=-z_{i,j}$; 
   
\item $Y=(y_{i,j})_{1\le i\le 2m,1\le j\le 2l}$ is a matrix for
which $$y_{i,j}=\begin{cases}y_{2m-i,j},&1\le i\leq m,1\leq j
\leq
l,\\y_{2m+1-i,j-l},&1\le i\le 2m,l+1\le j\le
2l.\end{cases}$$\end{enumerate}

Then 
$$\Pf(A)=(-1)^{\binom{l}{2}}\emph{det}(B),$$
where $B$ is an $(m+l)\times(m+l)$ matrix of the form
$$B=\begin{pmatrix}
   \widehat{X}&\widehat{Y}_1\\
   \widehat{Y}_2&\widehat{Z}
  \end{pmatrix},$$
the block matrices of which are defined by
\begin{align*}(\widehat{X})_{i,j}&=x_{i+j-1}+x_{i+j-3}+\dots+x_{|i-j|+1}&\textrm
{ for } 1\le i,j\le
m,\\(\widehat{Y}_1)_{i,j}&=\sum_{s=0}^{i-1}(y_{m+1-i+2s,j}-y_{
m+i-2s ,j} )&\textrm{for } 1 \le i \le m\textrm{ and }1\le j \le
l,\\(\widehat{Y}_2)_{i,j}&=\sum_{s=0}^{j-1}(y_{j+m-2s,i}+y_{m+1-j+2s
,i} )&\textrm{for } 1\le i \le l\textrm{ and } 1 \le j \le m,\\
(\widehat{Z})_{i,j}&=z_{i,j+l}+z_{i+l,j+l}&\textrm{for }1\le i,j\le l.
    \end{align*}
\end{lem}
\begin{proof}
 Beginning with $A$, construct a new matrix $A'=(A'_{i,j})_{1\le i,j\le 2m+2l}$
by simultaneously
replacing the $i$-th row of $A$ by the sum
$$\sum_{s=0}^{m-i}(\textrm{row } i + 2s \textrm{ of A}),$$
and the $(2m+1-i)$-th row of $A$ with
$$\sum_{s=0}^{m-i}(\textrm{row } 2m+1-i-2s \textrm{ of A})$$for $i=1,...,m-1$.
Perform analogous operations on the columns of the resulting matrix. Note
that these operations do not change the value of $\Pf(A)$ and so 
$\Pf(A)=\Pf(A')$. 
It follows that for $1\leq i,j\leq m$,
\begin{align*}A'_{i,j}&=\sum_{s=0}^{m-j}\sum_{r=0}^{m-i}x_{j-i+2s-2r}\\&=\sum_{
t=-m } ^ { m-i-j }
(\min\{ t+m+1 ,m-i+1\}-\max\{0,t+j\})x_{j+i+2t}.\end{align*}
By assumption $X$ is skew-symmetric, so $x_r=-x_r$ for all
$1\leq r \leq 2m$. One may also verify that
\begin{multline*}\min\{t+m+1,m-i+1\}\\-\max\{0,t+j\}=\min\{(-t-i-j)+m+1,m-i+1\}
\\-\max\{ 0 , (-t-i-j)+j\},\end{multline*}
whence $A'_{i,j}$ vanishes for $1\leq i,j\leq m$. It is not hard to convince
oneself
that replacing $i$ and $j$ with $2m-i+1$ and $2m-j+1$ respectively in the
above summation gives the $(i,j)$-entry of $A'$ for $m+1\le i, j \le 2m$,
and so these entries also vanish.

For $1\le i\le m$ and $m+1\le j \le 2m$,
\begin{align*}A'_{i,j}&=\sum_{s=0}^{j-m-1}\sum_{r=0}^{m-i}x_{j-i-2r-2s}\\&=\sum_
{ t=1 } ^ { j-i }
(\min\{t,m-i+1\}-\max\{0,t-j+m\})x_{i-j+2t}.\end{align*}
It is easy to see that 
\begin{multline*}\min\{(j-i-t),m-i+1\}\\-\max\{0,(j-i-t)-j+m\}=\min\{t,m-i+1\}
\\-\max\{0,t-j+m\}+1,\end{multline*}
so again by the skew-symmetry of $X$,
$$A'_{i,j}=\bar{x}_{i,j},$$
for $1\le i\le m$ and $m+1\le
j\le2m$, where
\begin{equation}\label{eqn:xbar}
\bar{x}_{i,j}=x_{j-i}+x_{j-i-2}+\dots+x_{|2m-i-j+1|+1}.\end{equation}

Having performed exactly the same operations on both the rows and columns of
$A$ it follows that $A'$ is also skew-symmetric, thus for $m+1\le
i\le 2m$ and $1\le j\le m$,
$A'_{i,j}=-A'_{j,i}.$

For the remaining $(i,j)$-entries of $A'$ it
suffices to consider only those entries for which $1\le i \le 2m$ and
$j\in\{2m+1,...,2m+l\}$. These columns are affected by the row operations alone
so for $1\le i\le m$ the $(i,j)$-entry of $A'$ is
$$\sum_{s=0}^{m-i}y_{i+2s,j-2m},$$
whilst an analogous argument shows that for $m+1\le i\le 2m$, $a'_{i,j}$ is
equal to
$$\sum_{s=0}^{i-m-1}y_{i-2s,j-2m}.$$

By the symmetry of $Y$ it should be clear that for $j\in\{2m+l+1,...,2m+2l\}$
and $1\le i\le m$,
$$A'_{i,j}=\sum_{s=0}^{m-i}y_{2m+1-i-2s,j-2m-l},$$
whilst for $m+1\le i\le2m$ and $j\in\{2m+l+1,...,2m+2l\}$ the $(i,j)$-entry of
$A'$ is
$$\sum_{s=0}^{i-m-1}y_{2m+1-i+2s,j-2m-l}.$$
Note that $A'$ has now been completely determined since these row and
column operations
leave $Z$ unchanged.

Construct a new matrix $A''=(A''_{i,j})_{1\le i,j\le 2m+2l}$ by perfoming the
following operations on $A'$: 
\begin{enumerate}[(i)]
 \item Add column $(2m+l+j)$ to column $(2m+j)$ for
$j\in\{1,...,l\}$ (resp. rows);
\item Subtract column $(m+j)$ from column $(m+1-j)$ for $j\in\{1,...,m\}$
(resp. rows).
\end{enumerate}
Consider the effect of the first set of row and column operations
stated above. For $1\le i,j \le2l$,
$$A''_{i+2m,j+2m}=\begin{cases}
                        0,& 1\le i,j \le l,\\
                        z_{i,j}+z_{i+l,j},&i\in\{1,...,l\},
j\in\{l+1,...,2l\},\\
z_{i,j}+z_{i,j+l},&i\in\{l+1,...,2l\},j\in\{1,...,l\},\\z_{i,j},
&l+1\le i,j\le 2l,
                       \end{cases}$$
by the second property in the statement of the lemma.

The $(i,j)$-entry of $A''$ for $1\le i
\le m$ and $2m+1\le j\le 2m+l$ becomes
$$A''_{i,j}=\sum_{s=0}^{m-i}y_{i+2s,j-2m}+\sum_{s=0}^{m-i}y_{2m+1-i-2s,j-2m}
,$$
while for $m+1\le i\le2m$,
$$A''_{i,j}=\sum_{s=0}^{i-m-1}y_{i-2s,j-2m}+\sum_{s=0}^{i-m-1}y_{2m+1-i+2s,
j-2m},$$
and crucially (by the third property of the statement of the lemma),
$A''_{i,j}=A''_{2m+1-i,j}$ for $1\le i\le m$ and $j\in\{2m+1,...,2m+l\}$.

Now consider the effect of the second set of operations applied to $A'$. Again
by the skew-symmetry of $A'$ this leaves
all entries unchanged except those for which $i\in\{1,...,m\}$ and
$j\in\{2m+1,...2m+2l\}$. For $2m+1\le j\le2m+l$ it should be clear that this
entry vanishes, hence the
resulting matrix has the form 
$$A''=\begin{pmatrix}
           0&\bar{X}&0&{Y}_1\\
           -\bar{X}^{t}&0&Y_{2}&{Y}_3\\
           0&-{Y^t_{2}}&0&Z_{1}\\
           -{{Y}^t_1}&-{{Y}^t_3}&-Z^t_{1}&Z_{2}
          \end{pmatrix},$$
where for $1\le i , j \le m$, $$(\bar{X})_{i,j}=\bar{x}_{i,j+m}$$ is
given by~\eqref{eqn:xbar} and
\begin{align*}(Y_1)_{i,j}&=\displaystyle\sum_{s=0}^{
m-i}(y_{
2m+1-i-2s,j}-y_{i+2s,j
})&\textrm{for
}i\in\{1,\dots,m\},j\in\{1,\dots,l\},\\(Y_2)_{i,j}
&=\displaystyle\sum_{s=0 } ^ { i-1 }( y_ { i+m-2s , j
}+y_{m+1-i+2s,j})&\textrm{for
}i\in\{1,...,m\},j\in\{1,...,l\},\\
(Y_3)_{i,j}&=\displaystyle\sum_{s=0}^{i-1}y_{
m+1-i+2s ,j}&\textrm{for
}i\in\{1,...,m\},j\in\{1,...,l\},\\(Z_{1})_{i,j}&=z_{i,j+l}+z_{i+l,j+l}&\textrm{
for }1\le i,j\le
l, \\(Z_ { 2 } )_{i,j}&=Z_{i+l,j+l}&\textrm{for } 1\le i,j\le
l.
\end{align*}
By rearranging rows and columns in exactly the same way $A''$ may be
brought into the form
\begin{equation}\begin{pmatrix}
           0&0&\bar{X}&{Y}_1\\
           0&0&-Y^t_{2}&Z_{1}\\
           -\bar{X}^t&{Y_2}&0&Y_3\\
           -{{Y}^t_1}&-Z^t_{1}&-{{Y}^t_3}&Z_{2}
          \end{pmatrix}.\end{equation}

Since the same operations have been performed on both rows and
columns this leaves the Pfaffian of $A''$ unchanged. By
the well-known identity
$$\Pf\begin{pmatrix}
              0&P\\
              -P^t&Q
             \end{pmatrix}
=(-1)^{\binom{n}{2}}\det( P),$$ where $P$ is an arbitrary $n\times n$ matrix,
it follows that
$$\Pf(A)=\Pf(A'')=(-1)^{\binom{m+l}{2}}\det
\begin{pmatrix}
 \bar{X}&Y_1\\
 -Y^t_2&Z_1
\end{pmatrix}.$$
Reversing the order of the rows $1$ to $m$ and multiplying the last $l$ rows
and columns by $-1$ gives
\begin{align*}\Pf(A)&=(-1)^{\left(\binom{m+l}{2}+\binom{m}{2}\right)}
\det\begin{pmatrix}
      \widehat{X}&\widehat{Y}_1\\
      \widehat{Y_2}&\widehat{Z}
     \end{pmatrix},\\&=(-1)^{\binom{l}{2}}\det\begin{pmatrix}
      \widehat{X}&\widehat{Y}_1\\
      \widehat{Y_2}&\widehat{Z}\end{pmatrix},\end{align*}
     where $\widehat{X},\widehat{Y}_1,\widehat{Y}_2,$ and
$\widehat{Z}$ are
exactly those blocks asserted in the lemma.
\end{proof}
\begin{rmk}\label{rmk:NoHole}
Gordon's Lemma may be recovered from the previous result by setting
$l=0$. 
\end{rmk}
Applying Lemma~\ref{lem:Gordon} directly to the matrix $F=(f_{i,j})_{1\le
i,j\le 2m+2}$ from Proposition~\ref{prop:VertMat} results in the following
expression for the signed Pfaffian of
$F$,
\begin{equation}\label{eqn:PfafF}\Pf(
F)=(-1)^{\binom{l}{2}}\det(\bar{F}),\end{equation}
where $\bar{F}=(\bar{f}_{i,j})_{1\le i,j\le m+1}$ is the matrix with
$(i,j)$-entries given
by
$$\bar{f}_{i,j}=\begin{cases}
                        f_{i+j-1}+f_{i+j-3}+\dots+f_{|i-j|+1},& 1\le
i,j\le m\\
\sum_{s=0}^{i-1}(f_{m+1-i+2s,2m+1}-f_{m+i-2s,2m+1}),&1\le
i\le m, j = m+1,\\
\sum_{s=0}^{j-1}(f_{j+m-2s,2m+1}+f_{m+1-j+2s,2m+1}),&
i=m+1,1\le j\le m,\\
0,&otherwise.
                       \end{cases}$$ 
Construct one final matrix $\widehat{F}=(\hat{f}_{i,j})_{1\le i,j\le m+1}$ by
performing one last set of
row
and column operations on $\bar{F}$, namely for $i\in\{1,..,m-1\}$,
subtract the ($m-i$)-th
row from row $(m+1-i)$, and once again perform analogous operations on the
columns. 

For $1\le j\le m$,
\begin{align*}\hat{f}_{m+1,j}&=\bar{f}_{m+1,j}-\bar{f
} _ { m+1 , j-1 }\\
&=f_{ m+j , 2m+1 } +f_{ m+1-j , 2m+1 }\\
&={\binom{n-k}{(n-k)/2+j}+\binom{n-k}{(n-k)/2+1-j}}\\&=\binom{n-k+1}{(n-k)/2+j},
\end{align*}
which agrees with $g^+_{m+1,j}$ (the entries of the matrix
whose determinant counts weighted tilings of
$\HPreg{n}{2m}{k}$) for $1\le j\le m$.

For $1\le i,j\le m$,
\begin{align*}\hat{f}_{i,j}&=\bar{f}_{i,j}-\bar{f}_{i-1,j}-(\bar{f}_{i,j-1}
-\bar{ f }_{ i-1 , j-1 } )\\
&=f_{i+j-1}-f_{i-j}-f_{i+j-2}+f_{i-j+1}\\&={2n\choose n-i-j+1}+{2n\choose
n+i-j},\end{align*}
which again agrees with $g^+_{i, j}$ for $1\le i,j\le m$.

The remaining entries of $\widehat{F}$ are given by
\begin{align*}\hat{f}_{i,m+1}&=\bar{f}_{i,m+1}-\bar{f}_{i-1,m+1}\\
&=f_{m+1-i,2m+1}-f_{ m+i,2m+1 } \\&\qquad+2\sum_{s=1}^{i-1}(f_{m+i+1-2s,2m+1}
                                    -f_{m-i+2s,2m+1 } ) \\
   &={ n-k\choose
(n-k)/2+1-i}-{n-k\choose
(n-k)/2+i}+2{\sum_{s=2-i}^{i-1}{n-k\choose
(n-k)/2+s}}\\
&=\left(\frac{4(i-1)}{(k-n)}+\frac{2
i-1}{i+(n-k)/2}\right)\binom{n-k}{(n-k)/2-i+1},
  \end{align*}
where the identity
$\sum_{k=0}^m\binom{N}{k}=(-1)^m\binom{N-1}{m}$ has been used in the fourth
line. Note how
none of these operations have changed the value of the determinant,
and also that $\hat{f}_{i,j}=g^+_{i,j}$ for $1\le i\le m+1$ and $1\le
j\le m$.

The sign in~\eqref{eqn:PfafF} may be ignored since these determinants are
considered within the context of counting families of non-intersecting paths,
and so
\begin{equation}M(\Vreg{n}{2m}{k})=|\det(\widehat{F})|,\end{equation}
where $\widehat{F}=(\hat{f}_{i,j})_{1\le i,j\le m+1}$ is the matrix given by
$$\hat{f}_{i,j}=\begin{cases}
\binom{2n}{n-i-j+1}+\binom{2n}{n+i-j}, &1\le i,j\le m,\\
\left(\frac{4 (i-1)}{k-n}+\frac{2
i-1}{i+(n-k)/2}\right)\binom{n-k}{(n-k)/2-i+1},&j=m+1,
1\le i\le m,\\
\binom{n-k+1}{(n-k)/2+j},&i=m+1,
1\le j\le m,\\0,&otherwise.\end{cases}$$

\begin{thm}\label{thm:VertLU}
 The $(m+1)\times(m+1)$ matrix $\widehat{F}$ described above has 
$LU$-decomposition
 $$\widehat{F}=\widehat{L}\cdot \widehat{U},$$
 where $\widehat{L}=(\hat{l}_{i,j})_{1\le i,j\le m+1}$ has the form
 $$\hat{l}_{i,j}=\begin{cases}
A_{n}(i,j),&1\le j\le i\le m,\\B_{n,k}(j),&i=m+1,1\le
j\le m\\
1,&i=j=m+1\\
0,&otherwise,
              \end{cases}$$
and $\widehat{U}=(\hat{u}_{i,j})_{1\le i,j\le m+1}$ has the form
$$\hat{u}_{i,j}=\begin{cases}
C_{n}(i,j),&1\le i\le j\le
m,\\D_{n,k}(i), &j=m+1,1\le
i\le m,\\
-\sum_{s=1}^{m}B_{n,k}(s)\cdot D_{n,k}(s), &i=j=m+1,\\
0,& otherwise,
              \end{cases}$$
where
\begin{align*}
A_{n}(i,j)&=\frac{( n)! (i+j-2)! (2 j+ n-1)!}{(2 j-2)! (i-j)! (-i+j+ n)!
(i+j+ n-1)!},\\
B_{n,k}(j)&=\frac{(-1)^{j+1} (j+ n-1)! (2 j+ n-1)! ( n-k+1)!
(j+(k+n)/2-2)!}{(j-1)! (2 j+2 n-1)! ((n-k)/2)! ((k+n)/2-1)! (j+(n-k)/2)!},\\
C_{n}(i,j)&=\frac{( n)! (i+j-2)! (2 i+2 n-1)!}{(j-i)! (2 i+ n-2)! (i-j+ n)!
(i+j+ n-1)!},\\
D_{n,k}(i)&=\frac{(-1)^{i+1} (2 i-2)! (i+ n-1)! ( n- k)!
(i+(k+n)/2-2)!}{(i-1)! (2 i+ n-2)! ((n-k)/2)! ((k+n)/2-1)!
(i+(n-k)/2)!}\\&\qquad+\frac{2
(-1)^{i+1} (2 i-2)! (i+ n)! ( n- k)! (i+(k+n)/2-2)!}{(i-2)! (2 i+ n-2)!
((n-k)/2)!
((k+n)/2)! (i+(n-k)/2)!}.\end{align*}

\end{thm}
\begin{proof}
 Before embarking on the proof proper observe that the $(m+1,m+1)$-entry of
$(\widehat{L}\cdot \widehat{U})$ is $0$, whilst for
$(i,j)\in\{1,...,m\}^2$, $$A_{n}(i,j)\cdot
C_{n}(i,j)=A_{n}(j,i)\cdot C_{n}(j,i).$$
The proof then amounts
to showing that for $1\le i,j\le m+1$,
$$\sum_{s=1}^{t}\hat{l}_{i,s}\cdot\hat{u}_{s,j}=\hat{f}_{i,j},$$
where $t=\min\{i,j\}$. 

This is equivalent proving that the following
identities hold:
\begin{enumerate}[(i)]
 \item $\sum_{s=1}^{i}A_{n}(i,s)\cdot C_{n}(s,j)={2n\choose
n-i-j+1}+{2n\choose n+i-j}$ for $1\le i\le j\le m$.
\item $\sum_{s=1}^{i}A_{n}(i,s)\cdot
D_{n,k}(s)=\left(\frac{4 (i-1)}{k-n}+\frac{2
i-1}{i+(n-k)/2}\right)\binom{n-k}{(n-k)/2-i+1}$ for $1\le i\le m.$
\item $\sum_{s=1}^{j}B_{n,k}(s)\cdot C_{n}(s,j)={n-k+1\choose
(n-k)/2+j}$ for $1\le j\le m.$
\end{enumerate}
In order to prove a hypergeometric identity of the form
\begin{equation}\label{eqn:Gsper}\displaystyle\sum_{k=1}^{n}U(n,k)=S(n,k),
\end{equation}
it suffices to find a suitable recurrence relation that is satisfied by both
sides of ~\eqref{eqn:Gsper}. Such a recurrence for the left hand side of
\eqref{eqn:Gsper} may
easily be found using your favourite implementation of the Zeilberger-Gosper
algorithm (see, for example~\cite{GspAlg}). Then verifying that the right hand
side also
satisfies such a recurrence and checking initial values reduces to routine
computations.

Identity (i) above satisfies the following recurrence
\begin{multline*}(-(i-j- n))
(i+j-n-1)\sum_{s=1}^{i}(A_{n}(i,s)\cdot
C_{n}(s,j))\\+(i^2+i-j^2+j-2n^2-n)\sum_{s=1}^{i+1}(A_{n}(i+1,s)\cdot
C_{n}(s,j))\\+(i-j+n+2)(i+j+n+1)\sum_{s=1}^{i+2}(A_{n}(i+2,s)\cdot
C_{n}(s,j))=0,\end{multline*}
whilst the second identity satisfies the recurrence
relation
\begin{multline*}(i+(k-n)/2-1)(2
i^2+2 i+(k-n)/2)\sum_{s=1}^{i}A_{n}(i,s)\cdot D_{n,k}(s)\\+(2 i^2-2
i+(k-n)/2) (i+(n-k)/2+1)\sum_{s=1}^{i+1}A_{n}(i+1,s)\cdot
D_{n,k}(s)=0.\end{multline*}
The third identity satisfies the recurrence
\begin{multline*}(j+(k-n)/2-1)\sum_{s=1}^{j}(B_{n,k}(s)\cdot
C_{n}(s,j)) \\+(j+(n-k)/2+1)\sum_{s=1}^{j+1}(B_{n,k}(s)\cdot
C_{n}(s,j+1))=0,\end{multline*}
which completes the proof.\end{proof}
\begin{proof}[Proof of Theorem~\ref{thm:VertExact} part (i)]
 It requires very little work to see that in the previous theorem
$A_{n}(i,i)=1$ for $i\in\{1,...,m\}$. Hence the matrix $\widehat{L}$ contains 1s
on
its diagonal, and so the determinant of $\widehat{F}$ is simply the product of
the diagonal entries of $\widehat{U}$, namely
$$\prod_{s=1}^{m}\frac{(n+2s)_{n}}{(2s-1)_{n}}\cdot\left(-\sum_{t=1}^{m}B_{n,k}
(t)\cdot
D_{n,k}(t)\right).$$ The product on the left hand side above is simply a
repackaging of $ST(n,2m)$ and so it follows that
$$M(\Vreg{n}{2m}{k})=\left[\sum_{s=1}^{m}B_{n,k}
(s)\cdot
D_{n,k}(s)\right]\times ST(n,2m,n).$$
\end{proof}
\begin{rmk}
By letting $k=0$, the above formula gives the number of vertically
symmetric tilings of $H_{n,2m}$ with no rhombi protruding across the vertical
line segment of length 2 that intersects the origin of $H_{a,b}$.
It appears that these kinds of tilings, where lozenges are forbidden from
crossing
segments of lattice lines, have not yet been considered in the literature. 
\end{rmk}

Given that $\hat{f}_{i,j}=g^+_{i,j}$ for $i\in\{1,\dots,m+1\}$ and
$j\in\{1,\dots,m\}$ (and also $\hat{f}_{m+1,m+1}=g^+_{m+1,m+1}=0$) a little
further work yields the following lemma.
\begin{lem}\label{thm:HorzWLU}
 The $(m+1)\times(m+1)$ matrix $G^+$ has $LU$-decomposition
$$G^+=\widehat{L}\cdot U^+,$$
where $\widehat{L}$ is defined according to Theorem~\ref{thm:VertLU} and the
matrix
$U^+=(u^+_{i,j})_{1\le i,j\le m+1}$ is given by
$$u^+_{i,j}=\begin{cases}\hat{u}_{i,j},&1\le i\le j\le m,\\
E_{n,k}(i),&1\le i\le m,j=m+1,\\
-\sum_{s=1}^{m}B_{n,k}(s)\cdot
E_{n,k}(s),&i=j=m+1.\end{cases},$$
where $E_{n,k}(s)$ is defined to be
$$E_{n,k}(s)=\frac{(-1)^{s+1} (2 s-2)! ( n-k+1)! ( n+s-1)!
((k+n)/2+s-2)!}{(s-1)! ((n-k)/2)! ((k+n)/2-1)! (n+2 s-2)! ((n-k)/2+s)!}.$$

\end{lem}
\begin{proof}
 It should be immediately clear that the $(m+1,m+1)$-entry of $(\widehat{L}\cdot
U^+)$ is $0$.
Given the aforementioned agreement between $g^+_{i,j}$ and $\hat{f}_{i,j}$ it
suffices to show that
$$\sum_{s=1}^iA_{n}(i,s)\cdot E_{n,k}(s)={n-k+1\choose (n-k)/2+i}.$$
A straightforward manipulation gives
$$A_{n}(i,s)\cdot E_{n,k}(s)=B_{n,k}(s)\cdot C_{n}(s,i),$$
so summing each side over $s$ from 1 to $i$ gives the result.
\end{proof}

\begin{proof}[Proof of Theorem~\ref{thm:HorzWExact}]Since this
$LU$-decomposition is again unique it follows that
\begin{equation}\label{eqn:Weight}
 \det(G^+)=\left[-\sum_{s=1}^{m}B_{n,k}(s)\cdot E_{n,k}(s)
\right]\times\prod_{t=1}^{m}\frac{(n+2t)_{n}}{(2t-1)_{n}}.
\end{equation}
The product on the right hand side of~\eqref{eqn:Weight} may be re-written so
that
\begin{align*}M(\HPreg{n}{2m}{k})&=\left[\sum_{s=1}^{m}B_{n,k}(s)\cdot
E_{n,k}(s)
\right]\times
ST(n,2m)\\&=\frac{1}{2}M(\HPreg{n+1}{2m-1}{k}
), \end{align*} thus concluding the proof.
\end{proof}

Consider now the skew-symmetric matrix $F^*=(f^*_{i,j})_{1\le i,j\le 2m+2}$
defined in Proposition~\ref{prop:VertMat}. Clearly this matrix does not satisfy
the conditions of Lemma~\ref{lem:Gordon} so instead a new matrix
$\xoverline{F^*}=(\bar{f^*}_{i,j})_{1\le i,j\le 2m+2}$ may be constructed
by performing the
following row and column operations on $F^*$:
\begin{enumerate}[(i)]
 \item Replace row $i$ with $\sum_{s=0}^{m-i}\textrm{row}(i+2s)$, and perform
analogous operations on the columns.
\item Subtract row $(2m+2)$ from row $(2m+1)$, and perform analogous
operations on the columns.
\end{enumerate}
In a similar fashion to the proof of Lemma~\ref{lem:Gordon}, for $1\le i, j\le
m$ the $(i,j)$-entry of $\xoverline{F^*}$ disappears, whilst for $1\le i\le m$
and $m+1\le j\le 2m-1$,
$$\bar{f^*}_{i,j}=\sum_{s=0}^{m-i}f^*_{i+2s,j}.$$
For $1\le i\le m$, $\bar{f^*}_{i,2m}=(m+1-i)2^n$ and
by observing that $f_{i,2m+1}^*=f_{2m-i,2m+2}^*$,

$$\bar{f^*}_{i,2m+1}=\sum_{s=0}^{m-i}\binom{n-k}{
(n-k-1)/2-m+i+2s }=\bar{f^*}_{i,2m+2}.$$
Hence for $1\le i\le m$, the entry
$\bar{f^*}_{i,2m+1}$ vanishes once the second set of row and column
operations have been applied, whilst for $m+1\le
i\le 2m-1$,
\begin{align*}\bar{f^*}_{i,2m+1}&=\binom{n-k}{(n-k+1)/2-m+i}-\binom{n-k}
{ (n-k-1)/2-m+i }\\&=\frac{2(m-i)}{n-k+1}\binom{n-k+1}{(n-k+1)/2-m+i}
.\end{align*}
Since $F^*$ is skew-symmetric and 
exactly the same row and column operations have been applied to $F^*$ in order
to construct
$\xoverline{F^*}$, the matrix $\xoverline{F^*}$ has now been
determined completely. By
reversing the order of the first $m$ rows and columns, and
then interchanging rows and columns in the correct way, $\xoverline{F^*}$ may
be brought into the form
$$\begin{pmatrix}0&P\\-P^T&*\end{pmatrix},$$
where $P=(P_{i,j})_{1\le i,j\le m+1}$ is the matrix with entries given by
$$P_{i,j}=\begin{cases}\sum_{s=0}^{i-1}f_{m+1-i+2s,j+m}^*,&,i\in\{1,\dots
,m\},j\in\{1,...,m-1\},\\2^n\cdot
i,&j=m,i\in\{1,...,m\},\\\sum_{s=0}^{i-1}\binom{n-k}{(n-k+1)/2-i+2s},
&j=m+1,1\le i\le
m,\\\frac{2j}{n-k+1}\binom{n-k+1}{(n-k+1)/2-j},&i=m+1,j\in\{1,\dots,
m\}.\end{cases}$$
Then clearly $|\Pf(\xoverline{F^*})|=|\det{(P)}|$. Proceed by
constructing a final matrix $\widehat{F^*}=(\hat{f}^*_{i,j})_{1\le i,j\le m+1}$
from $P$ by subtracting row $(m-i)$
from row $(m+1-i)$ for $i\in\{1,\dots,m-1\}$, and similarly for
$j\in\{2,\dots,m-1\}$ subtract column $(m-j)$ from column $(m+1-j)$.

For $i\in\{1,\dots,m\}$, $\hat{f}^*_{i,m}=2^n(i+1-i)=2^n$, while
\begin{align*}
\hat{f}^*_{i,m+1}&=\sum_{s=0}^{i-1}\binom{n-k}{(n-k+1)/2-i+2s}-\sum_{s=0}
^{i-2}\binom{n-k}{(n-k+3)/2-i+2s}\\&=\sum_{s=0}^{2i-2}(-1)^s\binom{n-k}{
(n-k+1)/2-i+s } \\&= \binom{n-k}{(n-k+1)/2-i},
\end{align*}
where the identity
$\sum_{k=0}^{m}(-1)^k\binom{A}{k}=(-1)^m\binom{A-1}{m}$ has been used in the
second line. For $1\le j\le m-1$,
$$\hat{f}^*_{m+1,j}=\frac{2j}{n-k+1}\binom{n-k+1}{(n-k+1)/2-j}-\frac{
2(j-1) }{n-k+1}\binom{n-k+1}{(n-k+1)/2-j+1},$$
and also $\hat{f}^*_{m+1,m}=P_{m+1,m}$. 

Lastly for $i\in\{1,\dots,m\}$ and $j\in\{1,\dots,m-1\}$,
\begin{align*}\hat{f}^*_{i,j}&=(P_{i,j}-P_{i,j-1
})-(P_{i-1,j}-P_{i-1,j-1})\\&=\sum_{s=0}^{i-1}(f^*_{m+1-i+2s,j+m}-f^*_{
m+1-i+2s ,
j+m-1})\\&\qquad-\sum_{s=0}^{i-2}(f^*_{m+2-i+2s,j+m}-f^*_{m+2-i+2s,j+m-1}
)\\&=\binom { 2n } {
n+i-j}+\binom{2n}{n-i-j+1}\\&=\hat{f}_{i,j}.\end{align*}

Thus it follows that
$$M(\Vreg{n}{2m-1}{k})=|\det(\widehat{F^*})|,$$
where $\widehat{F^*}=(\hat{f}^*_{i,j})_{1\le i,j\le m+1}$ is the matrix given by
by
$$\hat{f}^*_{i,j}=\begin{cases}\hat{f}_{i,j},&i\in\{1,\dots,m\},j\in\{1,\dots,
m-1\},\\2^n,&j=m,i\in\{1,\dots,m\},\\
\binom{n-k}{(n-k+1)/2-i},&j=m+1,1\le i\le m,\\
\frac{2j}{n-k+1}\binom{n-k+1}{(n-k+1)/2-j}-\frac{
2(j-1)
}{n-k+1}\binom{n-k+1}{(n-k+3)/2-j},&i=m+1,j\in\{1,\dots,m-1\},\\\frac{
2m}{n-k+1}\binom{n-k+1}{(n-k+1)/2-m},&i=m+1,j=m\\0,&otherwise.
\end{cases}$$
\begin{thm}\label{thm:VertOddLU}
The matrix $\widehat{F^*}$ defined above has $LU$-decomposition
$$\widehat{F^*}=L^*\cdot U^*$$ where $L^*=(l^*_{i,j})_{1\le i,j\le m+1}$ is
defined as
$$l^*_{i,j}=\begin{cases}
A_{n}(i,j),&1\le j\le i\le m,\\
B_{n,k}^*(j),&i=m+1, 1\le
j\le
m\\\frac{1}{D_{n}^*(m)}\left(-\sum_{
s=1 } ^ { m-1 } D_ { n } ^*(s)\cdot B_{n,k}^*(s)\right),&\i=m+1,j=m\\
1&i=j=m+1\\0,&otherwise.\end{cases}$$ and $U^*=(u^*_{i,j})_{1\le i,j\le m+1}$
is given by
$$u^*_{i,j}=\begin{cases}
                         C_{n}(i,j),&1\le i\le j\le m-1,\\
                         D_{n}^*(i),&j=m,1\le i\le m,\\
                         E_{n,k}^*(i),&j=m+1,1\le i\le m,\\
                         P_{n,k}^*(m),&i=j=m+1,\\
                         0,&otherwise,
                        \end{cases}$$
with $A_{n}(i,j)$ and $C_{n}(i,j)$ as before and
\begin{align*}
 B_{n,k}^*(j)&=\frac{2 (-1)^{j+1} (j+n-1)! (2 j+n-1)! (n-k)! \left(\frac{1}{2}
(2 j+k+n-5)\right)!}{(j-1)! (2 j+2 n-1)! \left(\frac{1}{2} (-k+n-1)\right)!
\left(\frac{1}{2} (k+n-3)\right)! \left(\frac{1}{2} (2
j-k+n+1)\right)!}\\&+\frac{4 (-1)^{j+1} (j+n)! (2 j+n-1)! (n-k)!
\left(\frac{1}{2} (2 j+k+n-5)\right)!}{(j-2)! (2 j+2 n-1)! \left(\frac{1}{2}
(-k+n-1)\right)! \left(\frac{1}{2} (k+n-1)\right)! \left(\frac{1}{2} (2
j-k+n+1)\right)!}\\
D_{n}^*(i)&=\frac{2^{n}(2 i-2)! (i+n-1)!}{(i-1)! (2 i+n-2)!},\\
E_{n,k}^*(i)&=\frac{(-1)^{i+1} (2 i-2)! (i+n-1)! (n-k)! \left(\frac{1}{2} (2
i+k+n-3)\right)!}{(i-1)! (2 i+n-2)! \left(\frac{1}{2} (-k+n-1)\right)!
\left(\frac{1}{2} (k+n-1)\right)! \left(\frac{1}{2} (2
i-k+n-1)\right)!}\\P_{n,k}^*(m)&=\sum_{s=1}^{m-1}\left(\frac{D_
{n}^*(s)E_{n,k }^*(m) } {D_{n}(m)}-E_{n,k}^*(s)\right)\cdot
B_{n,k}^*(s).
\end{align*}
\end{thm}
\begin{proof}
 It should be clear that the $(i,j)$-entry of $L^*\cdot U^*$ is $\hat{f}^*$
for
$(i,j)\in\{(m+1,m),(m+1,m+1)\}$ and since
$\hat{f}^*_{i,j}=\hat{f}_{i,j}$ for
$(i,j)\in\{1,\dots,m\}\times\{1,\dots,m-1\}$, proof of this theorem reduces to
showing that
the following identities hold:
\begin{enumerate}[(i)]
\item $\sum_{s=1}^{i}A_{n}(i,s)\cdot D_{n}^*(s)=2^n$ for $i\in\{1,\dots,m\}$;
\item $\sum_{s=1}^{i}A_{n}(i,s)\cdot E_{n,k}^*(s)=\binom{n-k}{(n-k+1)/2-i}$ for
$i\in\{1,\dots,m\}$;
\item $\sum_{s=1}^{j}B_{n,k}^*(s)\cdot
C_{n}(s,j)=\frac{2j}{n-k+1}\binom{n-k+1}{ (n-k+1)/2-j}-\frac{
2(j-1) }{n-k+1}\binom{n-k+1}{(n-k+3)/2-j}$ for $j\in\{1,\dots m-1\}$.
\end{enumerate}

Identity (i) may be written as a hypergeometric series (an
explanation as to how this may be done is postponed until
Section~\ref{sec:Asym})
\begin{multline}
\label{eqn:HypOdVert}\sum_{s=1}^{i}A_{n}(i,s)\cdot
D_n^*(s)=\pFq{4}{3}{n+1,\,
(n+3)/2,\,1-i,\,i}
{(n+1)/2,\,n+1+i,\,n-i+2}{-1}\times\\\frac{ 2^n(n+1)(n!)^2}{(n-i+1)!(n+i)!}.
\end{multline}
By applying the following summation formula
to~\eqref{eqn:HypOdVert} (see~\cite[(2.3.4.6); Appendix (III.10)]{Slat}),
$$\pFq{4}{3}{a,\, a/2+1, \,b,\, c}{a/2,\, a-b+1,\,
a-c+1}{-1}=\frac{\Gamma(a-b+1)
\Gamma(a-c+1)}{\Gamma(a+1) \Gamma(a-b-c+1)},$$
the right hand side of~\eqref{eqn:HypOdVert} may be transformed into
$$\frac{2^n(n+1)(n!)^2}{(n-i+1)!(n+i)!}\cdot\frac{\Gamma(n-i+2)
\Gamma(i+n+1)}{\Gamma(n+1)\Gamma(n+2)}=2^n.$$

Proof of the remaining identities follows in much the same way as
Theorem~\ref{thm:VertLU}.
The left hand side of identity (ii) satisfies the following recurrence
$$(2i+k-n-1)\sum_{s=1}^{i}A_{n}(i,s)\cdot
E_{n,k}^*(s)+(2i-k+n+1)\sum_{s=1}^{i+1}A_{n}(i+1,s)\cdot
E_{n,k}^*(s)=0,$$
whilst the left hand side of identity (iii) satisfies
\begin{multline*}
 (2 j + k - n - 3) (4 j^2 + 4 j + k - n - 1)\sum_{s=1}^{j}B_{n,k}^*(s)\cdot
C_{n}(s,j)+\\(4 j^2 - 4 j + k - n - 1) (2 j - k + n +
3)\sum_{s=1}^{j+1}B_{n,k}^*(s)\cdot
C_{n}(s,j+1)=0.
\end{multline*}
After verifying the right hand side of these identities also satisfy the
respective recurrence relations and checking initial conditions the proof is
complete.
\end{proof}
\begin{proof}[Proof of Theorem~\ref{thm:VertExact} part (ii)]
 Again it is easy to see that $l^*_{i,i}=1$ for $i\in\{1,\dots,m+1\}$ so it
follows that
$$\det(\widehat{F^*})=\left(\prod_{s=1}^{m}C_n(s,s)\right)\cdot D_n^*(m)\cdot
P_{n,k}^*(m).$$
The product above may be re-written as
$$M(\Vreg{n}{2m-1}{k})=P_{n,k}^*(m)\times ST(n,2m-1),$$
thus completing the proof of Theorem~\ref{thm:VertExact}.
\end{proof}

\begin{rmk}
Note that the product in Theorem~\ref{thm:VertExact} is a re-packaging of the
formula for the number of symmetric plane partitions in a
$n\times2m$ box. The appearance of MacMahon's original formula for
vertically symmetric tilings in Theorem~\ref{thm:VertExact} is not
surpising if one considers the translation of such plane partitions
to non-intersecting lattice paths. After a little thought one sees that
these may
be counted by the determinant of the $(m\times m)$ submatrix obtained by taking
rows (resp. columns) $1,\dots,m$ of the matrix $\widehat{F}$ (or indeed
$\widehat{F^*}$) defined above. The determinant of this submatrix is
simply $\prod_{i=1}^{m}\hat{u}_{i,i}=\prod_{i=1}^mu^*_{i,j}=ST(n,2m)$. 
\end{rmk}

What remains is to find an explicit formula for
the determinant of the matrix $G$ from Proposition~\ref{prop:HorzMat}.
The $LU$-decomposition of this matrix follows below.

\begin{thm}\label{thm:HorzLU}
 The matrix $G$ whose determinant counts the number of horizontally symmetric
tilings of $\Hreg{n}{2m}{k}$ has
$LU$-decomposition
$$G=L\cdot U,$$
where $L=(l_{i,j})_{1\le i,j\le m+1}$ is given by
$$l_{i,j}=\begin{cases}
                       A_{n}'(i,j),&1\le j\le i\le m,\\
                       B_{n,k}'(j),&i=m+1, 1\le j\le m,\\
                       1,&i=j=m+1\\
                       0,&otherwise,
                        \end{cases}$$
and $U=(u_{i,j})_{1\le i,j\le m+1}$ is given by

$$u_{i,j}=\begin{cases}
                     C_{n}'(i,j),&1\le i\le j\le m,\\
                     D_{n,k}'(i),&j=m+1,1\le i\le m,\\
-\sum_{s=1}^{m}B_{n,k}'(s)\cdot D_{n,k}'(s),
&i=j=m+1,\\0,&otherwise.
                        \end{cases}$$
                        where
                       
\begin{align*}
 A_{n}'(i,j)&=\frac{n! (i+j-2)! (2 j+n-1)! (2 i-1)}{(2 j-1)! (i-j)! (j-i+n)!
(i+j+n-1)!}, \\
 B_{n,k}'(j)&=\frac{(-1)^{j-1} (j+n-2)! (2 j+n-1)! (n-k)!
\left(j+\frac{k}{2}+\frac{n}{2}-2\right)!}{2 (j-1)! (2 j+2 n-3)!
\left(\frac{n}{2}-\frac{k}{2}\right)! \left(\frac{k}{2}+\frac{n}{2}-1\right)!
\left(j-\frac{k}{2}+\frac{n}{2}\right)!}, \\
C_{n}'(i,j)&=\frac{(2 j-1) n! (i+j-2)! (2 i+2 n-2)!}{(j-i)! (2 i+n-2)! (i-j+n)!
(i+j+n-1)!} , \\
D_{n,k}'(i)&=\frac{(-1)^{i+1} (2 i)! (i+n-1)!
 (n-k)!
\left(i+\frac{k}{2}+\frac{n}{2}-2\right)!}{2 (i!) (2 i+n-2)!
\left(\frac{k}{2}+\frac{n}{2}-1\right)! \frac{n-k}{2}!
\left(i-\frac{k}{2}+\frac{n}{2}\right)!}.
\end{align*}\end{thm}
\begin{proof}
 It should be immediately obvious that the $(m+1,m+1)$-entry of $(L\cdot U)$ is
0, after a little further thought one sees that $$A_{n}'(i,j)\cdot
C_{n}'(i,j)=A_{n}'(j,i)\cdot C_{n}'(j,i).$$
Straightforward manipulations give $$A_{n}'(p,q)\cdot
D_{n,k}'(q)=C_{n}'(p,q)\cdot B_{n,k}'(q),$$ hence it suffices to consider only
those entries of $L_g\cdot U_g$ for which $1\le i\le j\le m+1$.

The approach here is the same as in the
proof of Theorem~\ref{thm:VertLU},
that is, a set of recurrence relations are given that satisfy each side of the
following identities:
\begin{enumerate}[(i)]
 \item $\sum_{s=1}^{i}A_{n}'(i,s)\cdot C_{n}'(s,j)=\binom{2 n }{  n + j
- i }-\binom{2 n }{  n - j - i + 1}$;
\item $\sum_{s=1}^{i}A_{n}'(i,s)\cdot D_{n,k}'(s)=\binom{ n -  k }{ (n -
k)/2 +1- i}-\binom{ n -  k }{ (n -
k)/2 - i}$.\end{enumerate}
Again by using your favourite implementation of the Zeilberger-Gosper
algorithm it is possible to show that the left hand side of the first
identity above satisfies
the following recurrence
\begin{multline*}
 (j+n-i)(i+j-n-1)\sum_{s=1}^{i}A_{n}'(i,s)\cdot
C_{n}'(s,j)\\-2 (i^2+i-j^2-j-n^2-n)\sum_{s=1}^{i+1}A_{n}'(i+1,s)\cdot
C_{n}'(s,j)\\-(i+j+ n+1) (i-j+ n+2)\sum_{s=1}^{i+2}A_{n}'(i+2,s)\cdot
C_{n}'(s,j)=0,
\end{multline*}
whilst the second identity satisfies
\begin{multline*}
(2i+1)(2i+k-n-2)\sum_{s=1}^{i}A_{n}'(i,s)\cdot
D_{n,k}'(s)\\+(2i-1)(2i+n-k+2)\sum_{s=1}^{i+1}A_{n}'(i+1,s)\cdot
D_{n,k}'(s)=0. 
\end{multline*}
Verifying that the right hand side of (i) and (ii) satisfy the corresponding
recurrence relations and checking intial conditions is simply a routine
computation. This completes the proof.
\end{proof}
\begin{proof}[Proof of Theroem~\ref{thm:HorzExact}]
 Once again the $LU$-decomposition above is unique, it therefore follows that
 $$\det(G)=\left[-\sum_{s=1}^{m}B_{n,k}'(s)\cdot
D_{n,k}'(s)\right]\times\prod_{t=1}^{m}\frac{(n+2t-1)_{n}}{(2t)_{n}}.$$
The product on the right hand side may be re-written as Proctor's
formula~\cite{SymStan}, giving
\begin{align*}M(\HMreg{n}{2m}{k})&=\left[\sum_{s=1}^{m}B_{n,k}'(s)\cdot
D_{n,k}'(s)\right]\times TC(n,2m)\\&=M(\HMreg{n-1}{2m+1}{k}).\end{align*}
\end{proof}
\begin{rmk}Combining Theorem~\ref{thm:HorzExact} with
Theorem~\ref{thm:HorzWExact} by way
of Proposition~\ref{thm:MatchFac} yields Theorem~\ref{thm:HExact}.\end{rmk}

This concludes the proofs of Theorem~\ref{thm:VertExact},
Theorem~\ref{thm:HorzExact} and Theorem~\ref{thm:HExact}. The
asymptotics of the correlation functions defined in Section~\ref{sec:SetUp} may
now be derived using these exact enumeration formulas.

\section{Asymptotic Analysis}\label{sec:Asym}
This final section is devoted to the proofs of
Theorems~\ref{thm:CompCorr},~\ref{thm:VertCor}, and~\ref{thm:HorzCor}. The
exact enumeration formulas from Section~\ref{sec:EnumForm} are inserted into
the corresponding correlation functions defined in Section~\ref{sec:SetUp} and
asymptotic expressions for the different interactions are derived. For the sake
of simplicity the formulas that enumerate tilings of $\Vreg{a}{b}{k}$,
$\HPreg{a}{b}{k}$ and $\HMreg{a}{b}{k}$ for even $b$ are used in the
corresponding correlation functions, though of course the results obtained here
agree with the case for odd $b$.

Recall the correlation function $\omega_{V}(k;\xi)$ defined in
Section~\ref{sec:SetUp},
$$\omega(k;\xi)=\lim_{n\to\infty}\frac{M(\Vreg{n}{2m}{k})}{M(V_{n,2m})},
$$
where $m\thicksim\xi n/2$ and $M(R)$ denotes the number of tilings of the region
$R$. The enumeration formula for $M(\Vreg{n}{2m}{k})$ given in
Section~\ref{sec:EnumForm} lends itself
readily to asymptotic analysis of $\omega_V(k;\xi)$ since it is of the form
$M(\Vreg{n}{2m}{k})=M(V_{n,2m})\cdot\sum_{s=1}^{m} B_{n,k}(s)\cdot D_{n,k}(s)$.
It follows that $$\omega_V(k;\xi)=\lim_{n\to\infty}\left(\sum_{s=1}^{m}
B_{n,k}(s)\cdot
D_{n,k}(s)\right).$$Similar arguments hold for $\omega_{H^-}(k;\xi)$ and
$\omega_{H^+}(k;\xi)$. The product of these two functions then gives
$\omega_H(k;\xi)$ according to the Matchings Factorization
Theorem~\cite{MatchFac}.

In order to extract asymptotic expressions for these correlation functions the
finite sums from Theorems~\ref{thm:VertExact},~\ref{thm:HorzExact}
and~\ref{thm:HorzWExact} are first re-written as limits of
hypergeometric series. By applying some transformation formulas and sending
first $n$ and $m$ to infinity, one obtains these correlation functions in
terms of $k$ (which
parametrises the distance between the holes). Asymptotic expressions for the
different interactions are then derived by letting $k$ become large. Before
embarking on the proofs proper, some standard definitions and necessary
transformation formulas concerning hypergeometric series are recalled below.

Consider the $_pF_q$ hypergeometric series
$$\pFq{p}{q}{a_1,\,a_2,\,\dots,\,a_p}{b_1,\,b_2,\,\dots,\,b_q}{z}
=\sum_{s=0}^{
\infty } \frac{
(a_1)_s(a_2)_s\cdots(a_p)_s}{(b_1)_s(b_2)_s\cdots(b_q)_s}\frac{z^s}{s!}, $$
where $(\alpha)_{\beta}$ denotes the Pochhammer symbol, that is, 
$$(\alpha)_{\beta}=\begin{cases}(\alpha)_{\beta}
=(\alpha)(\alpha+1)\cdots(\alpha+\beta-1),&\beta\neq
0,\\1,&\beta=0.\end{cases}$$ 

A hypergeometric series is considered \emph{well-poised} if
$a_1+1=b_1+a_2=\cdots=b_q+a_p$. A \emph{very well-poised} hypergeometric series
is a well-poised hypergeometric series for which $a_2=a_1/2+1$.

It is often convenient to re-write a finite sum with summands that consist
of hypergeometric terms as a limit hypergeometric series by introducing a term
that
``artificially" terminates the series after a particular point. Consider, for
example, the following
finite
sum,
\begin{equation}\label{eqn:finsum}\sum_{s=0}^{t-1}=\frac{
(a_1)_s(a_2)_s\cdots(a_p)_s }
{(b_1)_s(b_2)_s\cdots(b_q)_s}\frac{z^s}{s!}.\end{equation}
 
By multiplying the summand by $(1-t)_s/(1-t+\epsilon)_s$ it is possible to
re-write this
finite sum as a limit of a hypergeometric series:

\begin{align*}\sum_{s=0}^{t-1}\frac{(a_1)_s(a_2)_s\cdots(a_p)_s}
{(b_1)_s(b_2)_s\cdots(b_q)_s}\frac{z^s}{s!}&=\lim_{\epsilon\to0}\left(\sum_{s=0}
^ { \infty}\frac{ (a_1)_s(a_2)_s\cdots(a_p)_s (1-t)_s}  {
(b_1)_s(b_2)_s\cdots(b_q)_s(1-t+\epsilon)_s}\frac{z^s}{s! } \right)\\&=\lim_ {
\epsilon\to0 } \left(\pFq { p+1 } { q+1 } {  a_1, 
a_2 ,
\dots,a_p,1-t+\epsilon,
}{b_1,b_2,\dots,\,d_q,1-t}{z}\right).\end{align*}

Suppose in addition that the terms $a_1,\dots,a_p$ and $b_1,\dots,b_q$ satisfy
the conditions for a very well-poised hypergeometric series. Then by multiplying
the summand in~\eqref{eqn:finsum} by a factor of
$$\frac{(1-t)_s(a_1+t)_s}{(a_1+t+\epsilon)_s(1-t+\epsilon)_s},$$ the finite
sum~\eqref{eqn:finsum} may be expressed
as a limit of a very well-poised $_{p+2}F_{q+2}$ hypergeometric series
\begin{multline*}\lim_{\epsilon\to0}\left(\sum_{s=0}
^ { \infty } \frac {
(a_1+\epsilon)_s(a_2+\epsilon/2)_s\cdots(a_p+\epsilon/2)_s(1-t)_s(a_1+t)_s } {
(b_1+\epsilon/2)_s(b_2+\epsilon/2)_s\cdots(b_q+\epsilon/2)_s(a_1+t+\epsilon)_s(1
- t +\epsilon) _ s }
\frac { z^s } { s! }\right)\\=\lim_{\epsilon\to0}\left(\pFq { p+2 } {
q+2 } { a_1+\epsilon, \,
\,a_2+\epsilon/2,\,\dots,\,a_p+\epsilon/2,\,1-t,\,a_1+t}{b_1+
\epsilon/2,
\, b_2+\epsilon/2 , \, \dots, \, b_q+\epsilon/2, \, a_1+t+\epsilon, \,
1-t+\epsilon } { z }\right)
.\end{multline*}

There are a wealth of transformation and summation formulas that may be applied
to hypergeometric series. Throughout this section frequent use shall be made of
the following transformation formula for a $_7F_6$ series
(see~\cite[(2.4.1.1), reversed]{Slat})
\begin{multline}\label{eqn:TransForm}\pFq{7}{6}{a,\,a/2+1,\,b,\,c,\,d,
\,e,\,-n}
{ a/2 , \,a-b+1 ,\, a-
c+1,\,a-d+1,\,a-e+1,\,a+n+1}{1}=\\\frac{(a+1)_n(a-d-e+1)_n}{(a-d+1)_n(a-e+1)_n}
\pFq{4}{ 3
}{a-b-c+1,\,d,\,e,\,-n}{a-b+1,\,a-c+1,\,-a+d+e-n}{1}.\end{multline}
Another transformation formula which will be useful is the following
(see~\cite[(1.8.10)]{Slat})
\begin{equation}\label{eqn:Trans2F1}
 \pFq{2}{1}{-n,\,a}{c}{z}=\frac{(1-z)^{n}(a)_n}{(c)_n} 
\pFq{2}{1}{-n,\,c-a}{-a-n+1}{(1-z)^{-1}}.
\end{equation}

\begin{proof}[Proof of Theorem~\ref{thm:VertCor}]
Consider the region $\Vreg{n}{2m}{k}$. According to Section~\ref{sec:SetUp} the
correlation function of this triangular hole with the free boundary to the
right of $\Vreg{n}{2m}{k}$ is
\begin{equation}\label{eqn:VSm}\omega_{V}(k;\xi)=\lim_{n\to\infty}\left(\sum_{
s=1 }^{ m } B_ { n , k } (s)\cdot D_{n,k}(s)\right),
\end{equation}
where $m\thicksim\xi n/2$ and
\begin{align*}
B_{n,k}(s)&=\frac{(-1)^{s+1} (-k+n+1)! (n+s-1)! (n+2s-1)!
\left(\frac{k}{2}+\frac{n}{2}+s-2\right)!}{(s-1)!
\left(\frac{n}{2}-\frac{k}{2}\right)! \left(\frac{k}{2}+\frac{n}{2}-1\right)! (2
n+2 s-1)! \left(-\frac{k}{2}+\frac{n}{2}+s\right)!},\\
 D_{n,k}(s)&=\frac{(-1)^{s+1} (2 s-2)! (n-k)! (n+s-1)!
\left(\frac{k}{2}+\frac{n}{2}+s-2\right)!}{(s-1)!
\left(\frac{n}{2}-\frac{k}{2}\right)! \left(\frac{k}{2}+\frac{n}{2}-1\right)!
(n+2s-2)! \left(-\frac{k}{2}+\frac{n}{2}+s\right)!}\\&\quad+\frac{2 (-1)^{s+1}
(2s-2)! (n-k)! (n+s)! \left(\frac{k}{2}+\frac{n}{2}+s-2\right)!}{(s-2)!
\left(\frac{n}{2}-\frac{k}{2}\right)! \left(\frac{k}{2}+\frac{n}{2}\right)! (n+2
s-2)! \left(-\frac{k}{2}+\frac{n}{2}+s\right)!}.
\end{align*}
The expression on the right of~\eqref{eqn:VSm} may be written as two
separate finite sums consisting of purely hypergeometric terms. These finite
sums may in turn be expressed as a limit of two very well-poised
hypergeometric series that have been artificially
terminated and so~\eqref{eqn:VSm} may be re-written in the following way:
\begin{multline}\label{eqn:Vh1}
 \omega_V(k;\xi)=\lim_{n\to\infty}\lim_{\epsilon\to0}\left(\frac{
n!(n+1)!(n-k)!(n-k+1)!}{
(2n+1)
!((n-k)/2)!^2((n-k)/2+1)!^2}\right.\\\times
\pFq { 7 }
{ 6 } {
n+1+\epsilon,\,\tfrac{n+3+\epsilon}{2},\,\tfrac{k+n+\epsilon}{2},\,
\tfrac{k+n+\epsilon}{2},
\tfrac{1+\epsilon}{2}, n+m+1,1-m } {\tfrac{n+1+\epsilon}{2},
\tfrac{n-k+4+\epsilon}{2},\tfrac{n-k+4+\epsilon}{2},\tfrac{2n+3+\epsilon}{2},
1-m+\epsilon ,n+m+1+\epsilon}{1}\\+
\frac{4(n+1)!(n+3)!((n+k)/2)!(n-k)!(n-k+1)!}{
(2n+3)!((n-k)/2)!^2((n-k)/2+2)!^2((n+k)/2-1)!}
\\\left.\times\pFq { 7 }
{ 6 } {
n+3+\epsilon,\tfrac{n+5+\epsilon}{2},\tfrac{k+n+2+\epsilon}{2},
\tfrac{k+n+2+\epsilon}{2} ,
\tfrac{3+\epsilon}{2}, n+m+2,2-m } {\tfrac{n+3+\epsilon}{2},
\tfrac{n-k+6+\epsilon}{2},
\tfrac{n-k+6+\epsilon}{2},\tfrac{2n+5+\epsilon}{2},
2-m+\epsilon,n+m+2+\epsilon }{1} \right).
\end{multline}
Applying~\eqref{eqn:TransForm} separately to each $_7F_6$ series in the above
sum gives
\begin{equation}\label{eqn:FB}
\frac{(n+2+\epsilon)_{m-1}((1+\epsilon)/2-m)_{
m-1}}{(n+(3+\epsilon)/2)_{ m-1} (1-m+\epsilon)_{ m-1}}
\pFq{ 4}{ 3}{2-k,\tfrac{1+\epsilon}{2},m+n+1,1-m}{\tfrac{
n-k+4+\epsilon}{2} , \tfrac{n-k+4+\epsilon}{2},
\tfrac{3-\epsilon}{2}} { 1 }
\end{equation}
for the upper hypergeometric series in~\eqref{eqn:Vh1} and
\begin{equation}\label{eqn:SB}
\frac{(n+4+\epsilon)_{m-2}((1+\epsilon)/2-m)_{m-2}}{
(n+(5+\epsilon)/2)_{m-2} (2-m+\epsilon)_{m-2} } \pFq{4}{3
}{2-k,
\tfrac{3+\epsilon}{2},m+n+2,2-m}{\tfrac{n-k+6+\epsilon}{2},
\tfrac{n-k+6+\epsilon}{2},\tfrac{5-\epsilon}{2} }{1}
\end{equation}
for the lower. Combining~\eqref{eqn:FB} with the corresponding
pre-factor in~\eqref{eqn:Vh1} and
letting $\epsilon$ tend to zero gives
gives
\begin{multline}\label{eqn:FB2}
\frac{(n-k)!(n-k+1)!(n+m)!(2m-1)!(n+m-1)!}{
((n-k)/2)!^2((n-k)/2+1)!^2(m-1)!^2(2n+2m-1)!}\times\\\pFq{4}{3}{2-k,\,
1
/2,\,m+n+1,\,1-m}{(n-k+4)/2,\,(n-k+4)/2,\,3/2}{1}.
\end{multline}
Applying Stirling's approximation $(n!\thicksim\sqrt{2\pi n}(n/e)^n)$ to the
pre-factor of~\eqref{eqn:FB2} as $n$ tends to infinity (and remembering that
$m\thicksim\xi n/2$) gives
$$\frac{(\xi(\xi+2))^{1/2} 4^{1-k}}{\pi  n},$$
which implies that~\eqref{eqn:FB2} vanishes in the limit. Hence the
correlation function is dominated by
the second sum~\eqref{eqn:SB}, chiefly
\begin{multline}\label{eqn:FB3}\omega_V(k;\xi)=\lim_{n\to\infty}\left(\frac{
2((n+k)/2)!(n-k)!(n-k+1)! } {
((n-k)/2)!^2((n-k)/2+2)!^2((n+k)/2-1)!}\right.\\\times
\frac{(2 m-1)! (m+n+1) (m+n-1)! (m+n)!}{3 (m-2)! (m-1)! 
(2 m+2 n-1)!}\\\left.\times\pFq{4}{3}{2-k,\,
3/2,\,m+n+2,\,2-m}{(n-k+6)/2,\,(n-k+6)/2,\,5/2}{1}\right),
\end{multline}
which follows from letting $\epsilon$ tend to zero in~\eqref{eqn:SB}.

In order to completely derive the asymptotics of~\eqref{eqn:FB3} it must be
shown the limit and sum operations commute in the following expression:
\begin{equation*}\lim_{n\to\infty}\left(\sum_{s=0}^{m}
 \frac{(2-k)_s(3/2)_s(m+n+2)_s(2-m)_s}{
(((n-k+6)/2)_s)^2(5/2)_s(s!) }\right).
\end{equation*}
As this series terminates for $s>k-2$, this is in fact a sum with a finite
upper bound dependent only on $k$ and so the limit and the sum in the above
expression may be safely interchanged. Doing so, and then applying Stirling's
approximation as $n$ tends to infinity, expression~\eqref{eqn:FB3} becomes 
\begin{equation}\label{eqn:SB2}\omega_V(k;\xi)=\frac{(\xi (\xi+2))^{3/2}
4^{1-k}}{3 \pi\cdot e
}\pFq{2}{1}
{ 2-k ,\, 3/2}{5/2}{-\xi(\xi+2)}.\end{equation}

To uncover the asymptotic behaviour of $\omega_V(k;\xi)$ as $k$ becomes very
large it is
first convenient to apply transformation~\eqref{eqn:Trans2F1} to the
hypergeometric series
in~\eqref{eqn:SB2}, thereby obtaining
\begin{equation}\label{eqn:SB3}
\omega_V(k;\xi)=\frac{(\xi (\xi+2))^{3/2} 4^{1-k}}{3 \pi
}\left(\frac{3(\xi+1)^{2k-4}}{2(k-1/2)}\pFq{2}{1}{2-k,\,1}{3/2-k}{
(\xi+1)^{-2}}
\right).
\end{equation}
Consider the hypergeometric series in $\omega_V(k;\xi)$ above, that
is,
$$\sum_{s=0}^{k-2}\frac{(2-k)_s}{(3/2-k)_s(\xi+1)^{2s}}.$$
It is easy to convince oneself that for any positive even $k$,
$\frac{(2-k)_r}{(3/2-k)_r}\le 1$ for all non-negative $s$ and so
\begin{equation}\label{eqn:lim}\lim_{k\to\infty}\sum_{s=0}^{k-2}\frac{(2-k)_s
} { (3/2-k)_s(\xi+1)^ { 2s } }
\le\lim_{k\to\infty}\sum_{s=0}^{k-2}\left(\frac{1}{(\xi+1)^{2}}\right)^s.
\end{equation}
Since $\xi>0$, the limit on the right hand side of~\eqref{eqn:lim} exists and
is
finite so the limit and sum operations on the left side of~\eqref{eqn:lim} may
be interchanged. Hence
as $k$ becomes very large,
$$\omega_{V}(k;\xi)\thicksim\frac{\sqrt{\xi(\xi+2)}}{2\pi
k}\left(\frac{(\xi+1)}{2}\right)^{2(k-1)},$$which concludes the
proof.\end{proof}

\begin{proof}[Proof of Theorem~\ref{thm:HorzCor}]
 Consider the correlation function
 $$\omega_{H^-}(k;\xi)=\lim_{n\to\infty}\left(\sum_{s=1}^{m}B_{n,k}'(s)\cdot
D_{n,k } '(s)\right),$$
where again $m\thicksim\xi n/2$ and
\begin{align*}
 B'_{n,k}(s)&=\frac{(-1)^{s+1} (n-k)! (n+s-2)! (n+2 s-1)! \left(\frac{1}{2}
(k+n+2 s-4)\right)!}{2 (s-1)! \frac{n-k}{2}! \left(\frac{1}{2} (k+n-2)\right)!
(2 n+2 s-3)! \left(-\frac{k}{2}+\frac{n}{2}+s\right)!},\\
D'_{n,k}(s)&=\frac{(-1)^{s+1} (2 s)! (n-k)! (n+s-1)!
\left(\frac{1}{2} (k+n+2 s-4)\right)!}{2( s!) \frac{n-k}{2}!
\left(\frac{1}{2} (k+n-2)\right)! (n+2 s-2)!
\left(-\frac{k}{2}+\frac{n}{2}+s\right)!}.
\end{align*}

The above finite sum may be re-written
as a limit of a hypergeometric series:
\begin{multline*}\label{eqn:Hhyp1}
\omega_{H^-}(k;\xi)=\lim_{n\to\infty}\lim_{\epsilon\to0}\left(\frac{
(n-1)!(n+1)!(n-k)!^2 } {
2(2n-1)!((n-k)/2)!^2((n-k)/2+1)!^2 }
\right.\\\left.\times\pFq{7}{6}{n+1+\epsilon,\tfrac{n+3+\epsilon}{2},
\tfrac{k+n+\epsilon}{2}, \tfrac{k+n+\epsilon}{2},
\tfrac{3+\epsilon}{2},n+m+1
, 1-m } {
\tfrac{n+1+\epsilon}{2} ,
\tfrac{n-k+4+\epsilon}{2} , \tfrac{n-k+4+\epsilon}{2}
,\tfrac{2n+1+\epsilon}{2},1-m+\epsilon,m+n+1+\epsilon}{1}\right).
\end{multline*}

Applying transformation formula~\eqref{eqn:TransForm} to the
hypergeometric series above and letting $\epsilon$ tend to zero gives
\begin{multline*}
\omega_{H^-}(k;\xi)=\lim_{n\to\infty}\left(\frac{(2 m+1)! ((n-k)!)^2 (m+n-2)!
(m+n)!}{12 (m-1)! m! \left(\frac{n-k}{2}!\right)^2 \left(\left(\frac{1}{2}
(-k+n+2)\right)!\right)^2 (2 m+2
n-3)!}\right.\\\left.\times\pFq{4}{3}{2-k,\,3/2,\, m+n+1,\,1-m}{(n-k+4)/2,\,
(n-k+4)/2,\, 5/2 }  {1} \right).
\end{multline*}
Consider the limit of this expression as $n$ tends to infinity. By interchanging
the sum
and limit of the hypergeometric series (as in the previous proof) and applying
Stirling's approximation to the pre-factor it may be shown that

$$\omega_{H^-}(k;\xi)=\frac{(\xi (\xi+2))^{3/2} 4^{1-k}}{3 \pi\cdot e
}\pFq{2}{1}{2-k,\,3/2}{5/2}{-\xi(\xi+2)},$$
which is precisely expression~\eqref{eqn:SB2} above.
\end{proof}

\begin{proof}[Proof of Theorem~\ref{thm:CompCorr}]
By the Matchings Factorization
Theorem~\cite{MatchFac} it suffices to consider the asymptotics of the
expression
$$\omega_{H^+}(k;\xi)=\lim_{n\to\infty}\left(\sum_{s=1}^{m}B_{n,k}(s)\cdot
E_{n,k}(s)\right),$$
where of course $m\thicksim\xi n/2$, $B_{n,k}(s)$ is as before and
$$E_{n,k}(s)=\frac{(-1)^{s+1} (2 s-2)! (-k+n+1)! (n+s-1)!
\left(\frac{k}{2}+\frac{n}{2}+s-2\right)!}{(s-1)!
\left(\frac{n}{2}-\frac{k}{2}\right)! \left(\frac{k}{2}+\frac{n}{2}-1\right)!
(n+2 s-2)! \left(-\frac{k}{2}+\frac{n}{2}+s\right)!}.$$

Again this correlation function may be written in as a limit:
\begin{multline}
\omega_{H^+}(k;\xi)=\lim_{n\to\infty}\lim_{\epsilon\to0}\left(\frac{
n!(n+1)!(n-k+1)!^2 } {
(2n+1)!((n-k)/2)!^2((n-k)/2+1)!^2 }\right.\\\left.\times\pFq{7 }{ 6
} {
n+1+\epsilon,\tfrac{n+3+\epsilon}{2},\tfrac{k+n+\epsilon}{2},
\tfrac{k+n+\epsilon}{2} , \tfrac{1+\epsilon}{2} ,
n+m+1,1-m } {\tfrac{n+1+\epsilon}{2},
\tfrac{n-k+4+\epsilon}{2}
, \tfrac{n-k+4+\epsilon}{2}
,\tfrac{2n+3+\epsilon}{2},1-m+\epsilon,m+n+1+\epsilon } { 1}\right).
\end{multline}
Transforming the above hypergeometric series according to~\eqref{eqn:TransForm}
and letting $\epsilon$ tend to zero gives
\begin{multline}\label{eq:hyp2}
\omega_{H^+}(k;\xi)=\lim_{n\to\infty}\left(\frac{(2 m-1)!  (n-k+1)!^2
(m+n)!(m+n-1)!}{(m-1)!^2 ((n-k)/2)!^2
((n-k)/2+1)!^2 (2 m+2 n-1)!}\right.\\\left.\times\pFq{4}{3}{2-k,\,1
/2,\,m+n+1,\,1-m}{(n-k+4)/2,\,(n-k+4)/2,\,3/2}{1}\right)
\end{multline}
As $n$ tends to infinity Stirling's approximation yields
\begin{equation}\label{eqn:Hphyp}\frac{\sqrt{\xi(\xi+2)} }{4^{k-1}\pi\cdot e
}\end{equation}
for the pre-factor while the hypergeometric series reduces to 
\begin{equation}\label{eq:hypp}\pFq{2}{1}{2-k,\,
1/2}{3/2}{-\xi(2+\xi)}.\end{equation}
This follows from interchanging the limit and the summand
in~\eqref{eq:hyp2}, again by a similar argument to that found in the proof of
Theorem~\ref{thm:VertCor}. Applying transformation~\eqref{eqn:Trans2F1}
to~\eqref{eq:hypp}, the correlation function $\omega_{H^+}(k;\xi)$ becomes
\begin{align*}\omega_{H^+}(k;\xi)&=\frac{(\xi+1)^{2k-4}}{2(k-3/2)}
\pFq { 2 } { 1 } { 2-k ,
\,1}
{5/2-k}{(\xi+1)^{-2}}\\&=\frac{1}{2}(\xi+1)^{2k-4}\sum_{s=0}^{k-2}\frac{
(k-1-s)_s } {
(k- s-3/2)_{s+1}(\xi+1)^{2s}}.\end{align*}
The asymptotic interaction of the holes is given by
\begin{equation}\label{eq:hhh}\lim_{k\to\infty}\omega_{H^+}(k;\xi)=\lim_{
k\to\infty}\left(\frac{1}{ 2}(\xi+1)^{2k-4}\sum_{s=0}^{k-2}\frac{
(k-1-s)_s } {(k- s-3/2)_{s+1}(\xi+1)^{2s}}\right).
\end{equation}

It should be clear that
$$\lim_{k\to\infty}\sum_{s=0}^{k-2}\frac{(k-1-s)_s}{
(k- s-3/2)_{s+1}(\xi+1)^{2s}}
\le\lim_{k\to\infty}\sum_{s=0}^{k-2}\left(\frac{1}{(\xi+1)^{2}}\right)^s,$$
so since the limit on the right exists and is finite the sum and limit
operations in~\eqref{eq:hhh} may be
interchanged.
Combining~\eqref{eq:hhh} with~\eqref{eqn:Hphyp} and letting $k$ become very
large gives
\begin{equation*}\omega_{H^+}(k;\xi)\thicksim\left(\frac{(\xi+1)}{2}\right)^{
2k-2
} \frac { 1
} { 2\pi k\sqrt{\xi(\xi+2)}}.\end{equation*}
Inserting this into Proposition~\ref{thm:MatchFac} along with the
result from Theorem~\ref{thm:HorzCor} yields
$$\omega_{H}(k;\xi)\thicksim\left(\frac { 1 } { 2k\pi }
\left(\frac{\xi+1}{2}\right)^{2k-2}\right)^2.$$
This completes the proof.
\end{proof}

\end{document}